\documentclass[reqno]{amsart}

\usepackage{amsmath}
\usepackage{amsthm}
\usepackage{amsfonts}
\usepackage{amssymb}
\usepackage{enumitem}
\usepackage[normalem]{ulem}
\usepackage{tikz}
\usepackage{tikz-cd}
\usepackage[colorlinks=true,linkcolor=blue,citecolor={blue}]{hyperref}
\usepackage{hyperref}
\usepackage[nocompress]{cite}{\normalsize }
\usepackage{mathrsfs}
\usepackage{subcaption}
\usepackage[only,llbracket,rrbracket]{stmaryrd}

\usepackage[right= 3cm, left=3cm, a4paper,top=3cm, bottom= 3cm]{geometry}

\newcommand{\oindicator}[1]{\ensuremath{\mathbf{1}_{{#1}}}}

\DeclareMathOperator{\sub}{\mathsf{sub}}

\DeclareMathOperator{\supp}{supp}
\DeclareMathOperator{\He}{He}
\DeclareMathOperator{\Ai}{Ai}
\newcommand{\dd}{\mathrm{d}}
\DeclareMathOperator{\ii}{i}

\newcommand{\C}{\mathbb{C}}

\renewcommand\Re{\operatorname{Re}}
\renewcommand\Im{\operatorname{Im}}
\newcommand{\eps}{\varepsilon}

\newcommand{\R}{\mathbb{R}}
\newcommand{\N}{\mathbb{N}}
\newcommand{\Z}{\mathbb{Z}}

\def\R{\mathbb{R}}

\newcommand{\LP}{\mathcal{LP}}

\theoremstyle{plain}
\newtheorem{theorem}{Theorem}[section]

\newtheorem{lemma}[theorem]{Lemma}
\newtheorem{corollary}[theorem]{Corollary}

\newtheorem{proposition}[theorem]{Proposition}

\theoremstyle{definition}
\newtheorem{definition}[theorem]{Definition} 
\newtheorem{assumption}{Assumption}

\newtheorem{problem}[theorem]{Problem}

\newtheorem{question}[theorem]{Question}

\theoremstyle{remark}
\newtheorem{remark}[theorem]{Remark}
\newtheorem{example}[theorem]{Example}

\numberwithin{equation}{section}

\begin{document}
	
	\title{P\'olya--Schur problems and free probability}
	
		\author{Andrew Campbell}
		\address[Andrew Campbell]{Institute of Science and Technology Austria, Am Campus 1, 3400 Klosterneuburg, Austria}
		\email[Andrew Campbell]{andrew.campbell@ist.ac.at}
		\thanks{AC is supported by the Austrian Science Fund (FWF) [DOI: 10.55776/ESP4314224]. For open access purposes, the author has applied a CC BY public copyright license to any author accepted manuscript version arising from this submission.}
		
		\author{Jonas Jalowy}
		\address[Jonas Jalowy]{Paderborn University, Institute of Mathematics, Warburger Str. 100, 33098 Paderborn, Germany}
		\email[Jonas Jalowy]{jjalowy@math.uni-paderborn.de}
		\thanks{JJ is supported by the DFG priority program SPP 2265 Random Geometric Systems.}
		

	\begin{abstract}
	In this work, we build a bridge between the P\'olya--Schur program and Voiculescu's free probability theory. A cornerstone of the former is the P\'olya--Benz Theorem, classifying a central family of real-root preserving operators on the space of polynomials, as those given by $f(\partial_z)$ for a Laguerre--P\'olya function $f$ and the derivative operator $\partial_{z}$. We prove that any free (additive) infinitely divisible distribution can be attained as the weak limit of root distributions of Appell polynomials $f_n(\partial_z)z^n$ as $n\to\infty$, for a suitably chosen sequence $f_n$ of Laguerre--P\'olya functions. 
	
	Such questions on the (global) limiting distributions of real rooted polynomials belong to the active research area of finite free probability.  In contrast to its standard tools, our approach allows for  non-compactly supported limiting distributions,  (barely) complex rooted polynomials and even provides the full microscopic description of the roots. Moreover, we extend our results to differential operators generating free multiplicative infinitely divisible distributions, to the rectangular free convolution, and to $f_n(\partial_z)p_n$ for real rooted polynomials $p_n$, implying a generalization of the recent connections between the heat flow and free Brownian motion to any free L\'evy process. 
	
	As corollaries, we identify free stable distributions by choosing $f_n$ to be a fixed rescaled Laguerre--P\'olya function, and we prove various convergence results on the zero distributions of Jensen polynomials, e.g.\ the limiting root distribution of Jensen polynomial of the Riemann $\Xi$-function is given by the Cauchy distribution. 
	\end{abstract}

	\maketitle

	\section{Introduction}
	\subsection{The Laguerre--P\'olya class}\label{subsec:LP class}
	Let $f$ be a real entire function of order less than or equal to $2$, of the form \begin{equation}\label{eq:func def}
		f(z)=Cz^{m}e^{-cz-\frac{\sigma^2}{2}z^2}\prod_{j=1}^{\infty}\left(1-\alpha_{j}z\right)e^{\alpha_{j}z}=\sum_{k=0}^{\infty}\frac{\gamma_{k}}{k!}z^k,
	\end{equation} where $c,\sigma\in\R$ and $ \{\alpha_{j}^{-1} \}_{j\geq1}\subset\R\times \ii[-M,M]$ for some $M\geq 0$ such that $\sum_j |\alpha_j|^2<\infty$, so that the infinite product is absolutely convergent. For any $M\geq0$ denote by $\LP_{M}$ the set of all such functions, and let $\LP_{\infty}:=\bigcup_{M\geq 0}\LP_{M}$. The set $\LP:=\LP_{0}$ is known as the \emph{Laguerre--P\'olya class} and it is the closure of hyperbolic polynomials under the topology of uniform convergence on compact subsets.   If $\sigma=0$, $\alpha_1,\alpha_{2},\dots\in[0,\infty)$, $\sum_{j}\alpha_{j}<\infty$, and $c\geq\sum_{j}\alpha_{j}$, then $f$ is in the \emph{Laguerre--P\'olya class of Type $\mathcal I$}, or simply $f\in\mathcal{LPI}$. Functions in $\mathcal{LPI}$ are those which are the uniform limits of polynomials with only non-negative roots.
	
	For a function $f\in\LP_{M}$, we consider two different infinite order differential operators, which we refer to as \emph{P\'olya--Benz type operator} \begin{equation}\label{eq:Polya-Benz type operator}
		f(\partial_{z})=\sum_{k=0}^{\infty}\frac{\gamma_{k}}{k!}\partial_{z}^{k},
	\end{equation} 
	and \emph{P\'olya--Schur type operator}
	\begin{equation}\label{eq:Polya-Schur type operators}
	f(z\partial_{z})=\sum_{k=0}^{\infty}\frac{\gamma_{k}}{k!}(z\partial_{z})^{k}.
	\end{equation} 
	We will study how the roots of $f(t\partial_z)p(z)$ and $f(tz\partial_{z})p(z)$ evolve for a polynomial $p$ and $t>0$.  These infinite differential operators and their effect on roots of polynomials and entire functions have been considered by various authors \cite{Craven-Csordas1994,Aleman-Beliaev-Hedenmalm2004,Kim--Kim2016,Cardon--deGaston2005} in the context of the \emph{P\'olya--Schur program}, which aims to understand linear operators which preserve geometric features of roots, more details in \S \ref{sec:polyaschur}. From this perspective, our work not only asks \emph{where} the zeros are distributed, but also \emph{how} they are distributed in the large degree limit.
	The names for the operators \eqref{eq:Polya-Benz type operator} and \eqref{eq:Polya-Schur type operators} are motivated by two classical results: The P\'olya--Benz Theorem (Proposition \ref{prop:Polya-Benz}) and P\'olya--Schur Theorem (Proposition \ref{prop:Polya-Schur}, see also \cite{Aleman-Beliaev-Hedenmalm2004}). 
 
	Let $\mathcal{P}(\C)$ denote be the space of polynomials with complex coefficients. We call $p\in\mathcal P(\C)$ \emph{real-rooted} if all its roots are real and say a linear operator $T:\mathcal{P}(\C)\rightarrow\mathcal{P}(\C)$ preserves real-rootedness if $T[p]$ is real-rooted for any real-rooted polynomial $p\in\mathcal P (\C)$.
	\begin{proposition}[P\'olya--Benz Theorem]\label{prop:Polya-Benz}
	Let $T:\mathcal{P}(\C)\rightarrow\mathcal{P}(\C)$ be a linear operator which commutes with $\partial_{z}$. Then, $T$ preserves real-rootedness if and only if $T=f(\partial_z)$ for some $f\in\LP$. 
	\end{proposition}

\subsection{Appell and Jensen polynomials}
	For any $f\in\LP_{\infty}$, we define the \emph{Appell polynomials} of $f$ by \begin{equation}\label{eq:Appell def}
		A_{n,f}(z)=f(\partial_{z})z^{n}=\sum_{k=0}^{\infty}\frac{\gamma_{k}}{k!}\partial_{z}^{k}z^{n},
	\end{equation} where we note the right hand side is a terminating series. 
	
	The  \emph{Jensen polynomials}, $J_{n,f}$, of $f$ are the reciprocal polynomial of $A_{n,f}$ given by\begin{equation}\label{eq:J and A inversion}
	J_{n,f}(z)=z^{n}A_{n,f}\left(\frac{1}{z}\right).
\end{equation} Given the power series expansion of $f$ in \eqref{eq:func def}, $A_{n,f}$ and $J_{n,f}$ have the simple formulas \begin{equation}\label{eq:formulas for A and J}
	A_{n,f}(z)=\sum_{k=0}^{n}\gamma_{k}\binom{n}{k}z^{n-k},\text{ and }
	J_{n,f}(z)=\sum_{k=0}^{n}\gamma_{k}\binom{n}{k}z^{k}.
\end{equation} A central feature of Appell sequences is the differential identity $A_{n,f}'=nA_{n-1,f}$. The exponential generating function of the Appell polynomials is given by \begin{equation}\label{eq:generating funciton}
f(z)e^{xz}=\sum_{k=0}^{\infty}A_{k,f}(x)\frac{z^k}{k!}.
\end{equation} We note that another common way to associate an Appell sequence to some analytic function $g$ is to through \eqref{eq:generating funciton} with $f(z)=\frac{1}{g(z)}$, see \cite{boyer2010appell,buric2016appell,salas2020appell,Anshelevich2004,Anshelevich2009,Anshelevich2009part3}.

\begin{proposition}\label{prop:Jensen polynomial convergence}
Let $f\in\LP_{\infty}$. Then, $f\in\LP$ if and only if $A_{n,f}$ has only real roots for every $n\geq 0$. Moreover,  uniformly on compact subsets, \begin{equation}
	\lim\limits_{n\rightarrow\infty}J_{n,f}\left(\frac{z}{n}\right)=f(z).
\end{equation}
\end{proposition} 

\subsection{Free probability}\label{subsec:free prob}

Voiculescu's free probability \cite{Voiculescu1986} is a non-commutative probability theory motivated by problems in operator algebras, and has since found great successes in a number of fields including random matrix theory and more recently the study of polynomial roots. 

\begin{definition}[Following \cite{Mingo-Speicher2017}]\label{def:R transform definition}
Let $\mu$ be a probability measure on $\R$. The \emph{Cauchy transform} $G_{\mu}:\C_{+}\to\C_-$ of $\mu$ is the analytic function 
\begin{equation}
		G_{\mu}(z):=\int_{\R}\frac{1}{z-t}\dd\mu(t).\label{eq:Cauchy}
	\end{equation} 
	 For $\alpha,\beta>0$, set $\Gamma_{\alpha,\beta}=\{z:\alpha\Im(z)>|\Re(z)|,\ \Im(z)>\beta \}$, and $\Delta_{\alpha,\beta}=\{z:z^{-1}\in\Gamma_{\alpha,\beta}\}$.
	 For every $\alpha>0$ there exists $\beta>0$ and the so-called $R$-\emph{transform} $R_{\mu}\colon\Delta_{\alpha,\beta}\rightarrow\C$ satisfying \begin{equation}\label{eq:Cauchy R relationship}
		G_\mu\left(R_{\mu}(z)+\frac{1}{z}\right)=z,\text{ for all }z\in\Delta_{\alpha,\beta},
	\end{equation}  and \begin{equation*}
		R_{\mu}\left(G_{\mu}(z)\right)+\frac{1}{G_{\mu}(z)}=z,\text{ for all }z\in\Gamma_{\alpha,\beta}.
	\end{equation*} 
\end{definition} 
With these tools at hand, it is possible to analytically introduce the following central concept of free probability.
\begin{definition}\label{def:free conv}
	For probability measures $\mu$ and $\nu$ on $\R$, the \emph{free additive convolution} $\mu\boxplus\nu$ is the unique probability measure such that $R_{\mu\boxplus\nu}=R_{\mu}+R_{\nu}$. 
	A probability measure $\mu$ is $\boxplus$\emph{-infinitely divisible} (or $\boxplus$-ID) if for any $n\in\N$, there exists a probability measure $\mu_n$ such that $\mu_{n}^{\boxplus n}=\mu$.
\end{definition}

 There exists the following classification of $\boxplus$-ID distributions, which serves as a \emph{free L\'evy--Khintchine formula}. \begin{proposition}[See \cite{Bercovici-Voiculescu1993}, Theorem 5.10 or \cite{Bercovici-Pata1999}, Theorem 3.2]\label{rthm:ID rep}
	A probability measure $\mu$ on $\R$ is $\boxplus$-ID if and only if $R_{\mu}$ can be analytically extended to the lower half-plane and there exists a positive finite Borel measure $\Sigma$ on $\R$ and a real number $c$ such that 
	\begin{equation}\label{eq:free levy-khin}
		R_{\mu}(u)=c+\int_{\R}\frac{u+x}{1-xu}\dd\Sigma(x).
	\end{equation}
\end{proposition}  We will refer to $\Sigma$ as the \emph{weighted L\'evy measure} of $\mu =\mu_{c,\Sigma}$ and $(c,\Sigma)$ as the \emph{L\'evy pair} of $\mu$.  
\begin{remark}\label{rem:weighted_levy}
	 The measure $\Sigma$ can be written in the form $\dd\Sigma(x)=\sigma^2\dd\delta_0(x)+\frac{x^2}{1+x^2}\dd\nu(x)$, where $\sigma\in\R$ and $\nu$ is a measure such that $\int_{\R}\min(x^2,1)\dd\nu(x)<\infty$ and $\nu(\{0\})=0$. It is common to refer to $(c,\sigma^2,\nu)$ as the \emph{L\'evy triple} of $\mu=\mu_{c,\sigma ^2,\nu}$, and $\nu$ as the \emph{L\'evy measure} of $\mu$.  
\end{remark}
An important subclass of $\boxplus$-ID distributions are the $\boxplus$-stable distributions.  A non-degenerate probability measure  $\mu$ is said to be $\boxplus$-stable with stability index $\alpha\in (0,2]$ if $\mu\boxplus \mu=\mathcal{D}_{2^{1/\alpha}}\mu\boxplus\delta_{b}$ for some $b\in\R$, where $\mathcal{D}_{2^{1/\alpha}}\mu$ is the dilation of $\mu$ by $2^{1/\alpha}$.  For $\alpha=2$, the stable law is a semicircle law with L\'evy pair $(c,\sigma^2\delta_0)$ and for $\alpha<2$ their weighted L\'evy measure has density
\begin{align}\label{eq:stable_levy_measure}
	\dd \Sigma_{\alpha,\theta}(x)=\alpha (\theta\oindicator{x>0}+(1-\theta)\oindicator{x<0} )\frac{|x|^{1-\alpha}}{1+x^2}\dd x 
\end{align}
for some $\theta\in[0,1]$.

The theory of \emph{finite free probability} studies zeros of real-rooted polynomials as the finite dimensional analogue of eigenvalue distributions in free probability theory, see \S \ref{sec:finite_free_probability} and \cite{Marcus-Spielman-Srivastava2015-1,Marcus-Spielman-Srivastava2015-2,Marcus-Spielman-Srivastava2022,Arizmendi-Perales2018,Jalowy-Kabluchko-Marynych2025} for more details. Using polynomial operations that mimic the free convolution, determining the limiting empirical root measures of $A_{n,f_n}$ is sufficient in order to determine the limiting empirical root measures of $f_n(\partial_{z})p_n(z)$ for any real--rooted $p_n$ and any $f_n\in\LP_{0}$. 
Recently, the backwards heat flow operator $e^{-\frac{t}{2n}\partial_z^2}p_n$ has been linked to the free convolution with the semicircle law and repeated differentiation $\partial_z^{\lfloor tn\rfloor}p_n$ has been linked to the free convolution semigroup, see for instance \cite{Arizmendi-GarzaVargas-Perales2023,Steinerberger2019,Steinerberger2020,Campbell-ORourke-Renfrew2024,Hall-Ho-Jalowy-Kabluchko2023heat,Hofert,Hoskins-Kabluchko2021,Hoskins-Steinerberger2022,campbell2025freeinfinitedivisibilityfractional,Kabluchko2022leeyang,Hall-Ho-Jalowy-Kabluchko2023}.
 
In this work, we unite and vastly generalize these previous results. 
\subsection{Summary of the results}

\begin{itemize}
	\item In our main result, we prove that any $\boxplus$-ID distribution $\mu_{c,\Sigma}$ can be obtained as the limiting empirical root measures of $f_n(\partial_z)z^n$ for some $f_n\in\mathcal{LP}$, whose reciprocal roots converge to $\Sigma$ after suitable rescaling (see Assumption \ref{assump:main assumption} and Theorem \ref{thm:main result, general}).
	\item These results continue to hold if the roots of $f_n$ are not too far from $\R$, even though $A_{f,n}$ has complex roots. We consider this as a step towards ``weakly non-Hermitian'' finite free probability.
	\item If the sequence functions is given by $f_n(z)=e^{-c_n{z}}f(b_nz)$ for a fixed function $f\in\LP_M$, then the limiting empirical measure is the spectral measure of a free \emph{stable} L\'evy process determined by the growth of the roots of $f$ at infinity (Corollary \ref{cor:fixed f and stable laws}). 
	\item As another corollary, we show that the law $\mu\boxplus {\mu_{c,\Sigma}}^{\boxplus t}$ of a \emph{free additive L\'evy process} starting from the distribution $\mu$ can be realized as the limiting empirical root measure of $f_{tn}(t\partial_{z})p_n(z)$ for some sequence of polynomials $p_n$ (Corollary \ref{cor:Free convolution theorem}).
	
	\item  We derive analogues for the \emph{multiplicative free convolution $\boxtimes$} and P\'olya--Schur type operators $f_n(z\partial_z)$, for more general operators $T_n$ commuting with $z\partial_z$, and for \emph{rectangular differentiation} $z\partial_{z}^2+(1+\beta)\partial_{z}$. 
	(Theorems \ref{thm:Polya schur and multiplicative groups}, \ref{thm:Polya-Schur full version} and \ref{thm:rectangular Appell polynomials}).
	
	\item  Our approach enables us to determine the microscopic structure of the roots of $f(\partial_{z})z^n$ in the bulk, whose lattice spacing is determined by the density of the limit $\mu_{c,\Sigma}$ and which generalizes the Plancherel–Rotach asymptotics for Hermite polynomials (Theorem \ref{thm:local spacing}).
	
	\item 
	The proof of Theorem \ref{thm:main result, general} relies on a novel combination of the saddle point method and insights from free probability, without tools from finite free probability (apart from some corollaries). We believe that this approach is of independent interest in order to remove typical restrictions of finite free probability, such as compact supports, real-rootedness and the lack of microscopic description.
\end{itemize} 	

The remainder of the paper is organized as follows. After fixing the notation, we rigorously present our main results in Section \ref{sec:mainresults}, first on P\'olya--Benz type operators then on P\'olya--Schur type operators and its rectangular sibling. We will provide an outline of the proofs in Section \ref{sec:outline_of_proof}. Section \ref{sec:Background} is devoted to a more detailed discussion of necessary concepts from (finite) free probability, the P\'olya--Schur program and related works. In Section \ref{sec:proof}, we prove our main result Theorem \ref{thm:main result, general}. Section \ref{sec:cors} presents proofs of the corollaries and Section \ref{sec:Polyaschur} deals with the P\'olya--Schur case. 



\subsection{Notation}
	For any polynomial $p$  of (formal) degree $n$  we define its \emph{empirical root distribution} $\llbracket p\rrbracket$ as \begin{equation}\label{eq:ERM}
		\llbracket p\rrbracket=\frac{1}{n}\sum_{z:p(z)=0}\delta_{z}.
	\end{equation} For $a\in\R$ we define the \emph{dilation} operator $\mathcal{D}_{a}$ on polynomials by \begin{equation}
	\mathcal{D}_{a}p(x)=a^{\deg(p)}p\left(\frac{x}{a}\right).
\end{equation} This operator dilates the roots of $p$ by a factor of $a$ while preserving the leading coefficient. We use the same notation to define the operator which dilates the support of a measure by a factor of $a$, i.e.\ \begin{equation}
\mathcal{D}_{a}\mu(B)=\mu\left(\frac{B}{a}\right).
\end{equation} 


We distinguish two types of convergences of measures by denoting $\Rightarrow$ for weak convergence and $\rightarrow$ for vague convergence of measures.

We will use the notation $\C_-=\{u\in\C: \Im u<0\}$ and $\C_{+}=\{u\in\C:\Im u>0 \}$ to denote the lower and upper half-planes, respectively, as well as $\R_-=(-\infty,0)$, $\R_+=(0,\infty)$. 
For an analytic function $g$ on the upper half-plane and $z\in\R$, we will use the notation $g(z+\ii0)$ to denote the continuous extension to the real line, when such an extension exists. 

We will use the notation $O(\cdot), o(\cdot),$ throughout as well as $b_n\sim c_n$ to denote asymptotic equivalence ${b_n}={c_n} (1+o(1))$ as $n\rightarrow\infty$.

\section{Main results}\label{sec:mainresults}
We will assume the following condition on a sequence of functions in $\LP_{\infty}$ for all our results. 
\begin{assumption} \label{assump:main assumption}
	Let $0\leq M_n=o(n)$, and let $\{f_{n}\}_{n=1}^{\infty}$ be such that $f_n\in \LP_{M_n}$ and of the form \begin{equation}
		f_n(z)=e^{-c_{n}z-\frac{\sigma_{n}^2}{2}z^2}\prod_{j=1}^{\infty}\left(1-\alpha_{j,n}z\right)e^{\alpha_{j,n}z}.
	\end{equation} 
	We say the sequence $\{f_n\}$ satisfies Assumption \ref{assump:main assumption} if the finite complex valued measures
	\begin{equation}\label{eq:Sigma_n assumptions}
		\Sigma_{n}:=n\sigma_{n}^2\delta_{0}+\frac{1}{n}\sum_{j=1}^{\infty}\frac{n^2\alpha_{j,n}^2}{n^2\alpha_{j,n}^2+1}\delta_{n\alpha_{j,n}} \to \Sigma
	\end{equation}
	converge vaguely to some positive weighted L\'evy measure $\Sigma$ as $n\to\infty$, and if 
	\begin{equation}\label{eq:c assumptions} \lim\limits_{n\rightarrow\infty}c_{n}-\int t\dd\Sigma_{n}(t)=c.\end{equation} 
	\end{assumption}

   Since the L\'evy pair is continuous (as real numbers and vaguely as finite positive measures) with respect to weak convergence of probability measures, Assumption \ref{assump:main assumption} is a continuity statement about a sequence of $\boxplus$-ID measures in the case when $f_{n}\in\LP$ and $\Sigma_n$ is a positive measure. Allowing for complex roots creates some technical obstacles in our proofs, which $\Sigma_{n}$ being a complex measure is a relatively minor one. By the Hahn--Jordan decomposition this complex measure can be written as a linear combination of four finite positive measures \begin{equation*}
   	\Sigma_n=\Sigma_{n,+}-\Sigma_{n,-}+\ii\Sigma_{n,+\ii}-\ii\Sigma_{n,-\ii}.
   \end{equation*} Implicit in \eqref{eq:Sigma_n assumptions} is that three of these measures converge vaguely to the zero measure. 
   Assumption~\ref{assump:main assumption} can easily be rephrased in terms of the triple $c_{n}, \sigma_{n}$ and the positive infinite measure $\nu_n:=\frac{1}{n}\sum_{j}\delta_{n\alpha_{j}}$ on $\C$ to avoid complex valued measures entirely. However, it turns out that our choice of working with \eqref{eq:Sigma_n assumptions} and \eqref{eq:c assumptions} will keep the proofs cleaner and less technical. 
   
   Assumption \ref{assump:main assumption} should be viewed as an optimal condition for our results and we will explicitly discuss necessity in Remark \ref{rem:optimality} below.



\subsection{Results for P\'olya--Benz type operators} We will begin with $\llbracket A_{n,f_n}\rrbracket$ when $f_{n}(0)\neq 0$. 
\begin{theorem}\label{thm:main result, general}
If $\{f_{n}\}$ satisfy Assumption \ref{assump:main assumption}, then weakly as probability measures  \begin{equation}
	\llbracket A_{n,f_n}\rrbracket\Rightarrow \mu_{c,\Sigma},
\end{equation} 
where $\mu_{c,\Sigma}$ is the $\boxplus$-infinitely divisible distribution with L\'evy pair $(c,\Sigma)$. 
\end{theorem}
%

The proof of Theorem \ref{thm:main result, general} relies on a novel combination of the saddle point method and free probability identities,  and we will outline it in Section \ref{sec:outline_of_proof}.
There, we will explain that $-\frac{f_n'(nu)}{f_n(nu)}$ is ideally the $R$-transform of an $\boxplus$-ID measure $\mu_{n}$, which is close to $\llbracket A_{n,f_n}\rrbracket$ and converges to $\mu_{c,\Sigma}$. For $f_n\in\LP$ and compactly supported measures on $\R$, this approximation can be made rigorous using moments,  see \cite{campbell2025freeinfinitedivisibilityfractional,Arizmendi-Fujie2026}, but we are especially interested in (stable) laws which have very few moments requiring new techniques.

In order to get a better picture of Theorem \ref{thm:main result, general}, let us discuss some consequences. We immediately have the following corollary. 
\begin{corollary}\label{cor:Convergnce of Jensen polynomials}
	Let $\{f_n\}_{n\geq 1}$ be a sequence of entire functions satisfying Assumption \ref{assump:main assumption}. Then, \begin{equation}
		\llbracket J_{n,f_n}\rrbracket\Rightarrow\mu_{c,\Sigma}^{-1},
	\end{equation} where $\mu_{c,\Sigma}^{-1}$ is the push-forward of the measure $\mu_{c,\Sigma}$ by the map $x\mapsto x^{-1}$.
\end{corollary}

One particularly illuminating special case of Theorem \ref{thm:main result, general} is the case when $f_n$ is defined via a fixed function $f$, up to a shift and normalization of the roots of $A_{n,f}$.
 Then, Assumption \ref{assump:main assumption} becomes a condition on the behavior of the root density of $f$ near $\pm\infty$.  For any $f\in\LP_{M}$  we define
\begin{align*}
n_{f,+}(r)&= |\{j\in\N:\alpha_{j}^{-1} \in (0,r)\times\ii[-M,M]\}|\\
n_{f,-}(r)&= |\{j\in\N:\alpha_{j}^{-1} \in (-r,0)\times\ii[-M,M]\}|
\end{align*} 
 and $n_{f}(r)=n_{f,+}(r)+n_{f,-}(r)$.  Let \begin{equation}\label{eq:b_n def}
 	b_n=\inf\{b>0:n_{f}(b)\geq n\},
 \end{equation} 
 be the pseudo-inverse and, for any  real sequence $E_n\to E\in\R$,  define 
 \begin{equation}\label{eq:a_n}
 	a_n=E_n-c\frac{b_n}{n}+\frac{1}{n}\sum_{j=1}^{\infty}\frac{(b_n\alpha_{j})^3}{1+(b_n\alpha_{j})^2},
 \end{equation} 
 where $c\in\R$ comes from the definition \eqref{eq:func def} of $f$.    
 A function $g:(0,\infty)\rightarrow(0,\infty)$ is said to be \emph{regularly varying} at $\infty$ with  index $\alpha$ if for every $x>0$ \begin{equation}\label{eq:RV def}
	\lim\limits_{r\rightarrow\infty}\frac{g(xr)}{g(r)}=x^{\alpha}.
\end{equation} If the index $\alpha=0$, then $g$ is said to be \emph{slowly varying}. A classical result in the study of regular variation (see for example \cite{Resnick2007}) is that if $g$ is regularly varying with index $\alpha$, then there exists a slowly varying function $h$ such that $g(r)=h(r)r^{\alpha}$. 

\begin{corollary}\label{cor:fixed f and stable laws}
	Let $f\in\LP_{M}$ for some $M\geq 0$. \begin{enumerate}
		\item If $\sigma^2>0$ and $f_n(z)=f(n^{-1/2}z)$, then \begin{equation}
			\llbracket A_{n,f_n}\rrbracket\Rightarrow \mathsf{sc_{\sigma}},
		\end{equation} the semicircle law with variance $\sigma^2$. 
		
		\item If $\sigma^2=0$, $n_f$ is regularly varying with index $\alpha\in (0,2)$ and the limit\begin{equation}\label{eq:theta def}
			\lim\limits_{r\rightarrow\infty} \frac{n_{f,+}(r)}{n_{f}(r)}=\theta\in[0,1],
		\end{equation} exists,  then set $f_n(z)=e^{-a_nz}f\left(\frac{b_nz}{n}\right)$ to obtain
		\begin{equation}
		\llbracket A_{n,f_n}\rrbracket\Rightarrow\delta_{E}\boxplus\mu_{\alpha,\theta},
		\end{equation}  where $\mu_{\alpha,\theta}$ is the free $\alpha$-stable distribution with L\'evy pair $(0,\Sigma)$ as defined in \eqref{eq:stable_levy_measure}.
	\end{enumerate}
\end{corollary} 
 If we extend the pseudo-inverse $b_n$ to a function $b:\R_+\rightarrow\R_+$ given by $b(s)=\inf\{b>0:n_f(b)\geq s\}$, then this function is also regularly varying at infinity with index $\frac{1}{\alpha}$, see for example \cite{Resnick2007}. Thus, the scale on which the empirical root measures converge is given either by the Gaussian component $\sigma^2>0$ or by the index $\alpha$.   

The most classical example of part (1) in Corollary \ref{cor:fixed f and stable laws} is given by the Gaussian $f(x)=\exp(-\sigma^2x^2/2)$ with corresponding Appell-polynomials $A_{n,f_n}(x)=\frac{\sigma ^n}{n^{n/2}}\He_n(\sqrt n x/\sigma)$, where $\He_n$ are the probabilist's Hermite polynomials whose root asymptotics are well known to follow the semicircle law $\mathsf{sc_{\sigma}}$, see for example \cite{Jalowy-Kabluchko-Marynych2025part2}. Examples of functions for which Corollary \ref{cor:fixed f and stable laws} applies include the the Airy function $\Ai$ ($\alpha=3/2, \theta=0$), the reciprocal Gamma function ($\alpha=1, \theta=0$), $\cos(z)$ and $\sin(z)/z$ ($\alpha=1, \theta=1/2$), and many others. Corollary \ref{cor:fixed f and stable laws} also applies almost surely to random entire functions whose roots are distributed to some appropriate point process, such as the \emph{stochastic $\zeta$-function} ($\alpha=1, \theta=1/2$) \cite{Valko-Virag2022}, a function with homogeneous Poisson roots such as in \cite{Pemantle-Subramanian2017} ($\alpha=1, \theta=1/2$), or more general random $\LP$ functions as considered in \cite{Assiotis2022}. 
Let us explicitly mention the following special case of Corollaries \ref{cor:Convergnce of Jensen polynomials} and \ref{cor:fixed f and stable laws}, which is one of our motivations for allowing functions $f\in\LP_M$ with roots in a strip.  \begin{corollary}\label{cor:Riemann Xi}
	Let $f_{n}(z)=\Xi(2\log(n)^{-1} z)$, where $\Xi(z)=\xi\left(\frac{1}{2}+iz\right)$, \begin{equation}
		\xi(z)=\frac{1}{2}z(1-z)\pi^{-z/2}\Gamma\left(\frac{z}{2}\right)\zeta(z),
	\end{equation} and $\zeta$ is the Riemann $\zeta$-function. Then, \begin{equation}
		\llbracket A_{n,f_n}\rrbracket\Rightarrow\frac{1}{\pi}\frac{1}{x^2+1}\dd x.
	\end{equation} Since the Cauchy distribution is invariant under $x\mapsto\frac{1}{x}$, it also follows that \begin{equation}
		\llbracket J_{n,f_n}\rrbracket \Rightarrow \frac{1}{\pi}\frac{1}{x^2+1}\dd x.
	\end{equation} 
\end{corollary} The Riemann Hypothesis is equivalent to the statement that $\Xi\in\LP$. However, all that can currently be said about $\Xi$ is that all its roots are contained in a strip.  

So far, all our results on Appell polynomials can be understood as the operator $f_n(\partial_z)$ applied to the monomial $z^n$ with initial root distribution being the Dirac delta $\delta_0$. The following extends this to a dynamic differentiation flow, yielding a time-dependent distribution of a free L\'evy process with any given initial distribution $\mu$. 

\begin{corollary}\label{cor:Free convolution theorem}
	Let $\{f_n\}_{n=1}^{\infty}\subset\LP$ and satisfy Assumption \ref{assump:main assumption}. If $t\geq 0$ and $\{p_n\}_{n\geq 1}$ is a sequence of real-rooted polynomials indexed by their degree such that $\llbracket p_n\rrbracket\Rightarrow\mu$ weakly as $n\rightarrow\infty$, then \begin{equation}
		\llbracket f_{\lfloor tn\rfloor}\left(t\partial_{z}\right)p_n(z)\rrbracket\Rightarrow \mu\boxplus\mu_{c,\Sigma}^{\boxplus t},
	\end{equation} where $\boxplus$ is the free convolution of probability measures.
\end{corollary}
 Corollary \ref{cor:Free convolution theorem}  extends Kabluchko's \cite{Kabluchko2022leeyang} result on the \emph{backwards heat flow} from convergence to the distribution of free Brownian motion to any free L\'evy process. If we consider the special case when $f_{n}(u)=e^{-c_nu}f\left(\frac{b_n}{n}u\right)$ in Corollary \ref{cor:fixed f and stable laws}, then the limit is the law of a free $\alpha$-stable process. 

For the previous results we assumed that $f_n(0)=1$. If $f_n$ has a root of multiplicity $m_n$ at $0$, then this amounts to differentiating polynomials whose limiting empirical root measure we understand. This problem has been considered by a number of authors in recent years \cite{Steinerberger2019,Steinerberger2020,Arizmendi-GarzaVargas-Perales2023,Hoskins-Kabluchko2021,Campbell-ORourke-Renfrew2024,Arizmendi-Fujie-Perales-Ueda2024,Hall-Ho-Jalowy-Kabluchko2023Repeat,Arizmendi-Campbell-Fujie2025}.  

\begin{corollary}\label{cor:f with a root at 0} 
	Let $\tilde{f}_n(z)=z^{m_n}f_n(z)$ for a sequence of $f_n\in\LP$ satisfying  Assumption \ref{assump:main assumption}. If $\{p_n\}_{n\geq 1}$ is a sequence of real-rooted polynomials indexed by their degree such that $\llbracket p_n\rrbracket\Rightarrow\mu$ and $\frac{m_n}{n}\rightarrow t\in [0,1)$, then \begin{equation}
		\llbracket \tilde{f}_{n}(\partial_z)p_{n}\rrbracket\Rightarrow\mathcal{D}_{1-t}\left(\mu\boxplus\mu_{c,\Sigma} \right)^{\boxplus\frac{1}{1-t}}.
	\end{equation}
\end{corollary}

We also obtain the following full local description of the roots of $A_{n,f_n}$.

\begin{theorem}\label{thm:local spacing}
	Let $\{f_n\}$ satisfy  Assumption \ref{assump:main assumption} and $E\in\R$ be such that $\mu_{c,\Sigma}$ has bounded density at $E$. Then, there exists constants $C_{E,n},B_{E,n}$ depending on $E,\mu_{c,\Sigma}$, and $n$, and $\phi_E\in [-\pi,\pi]$ depending on $E$ and $\mu_{c,\Sigma}$ such that for any fixed $d\in\N$ and as $n\to\infty$ we have \begin{align*}
		A_{n+d,f_n}&\left(E+\frac{w}{n}\right) \sim C_{E,n}|G(E)|^{-d}e^{w\Re G(E)}\cos\left(B_{E,n}+w\Im G(E)-d\arg G(E)\right) , 
	\end{align*} uniformly in $w$ on compact subsets of $\R$, where $G(E)=G(E+0\ii)$ is the continuation of the Cauchy transform of $\mu_{c,\Sigma}$ to $\R$, defined in \eqref{eq:Cauchy}.
\end{theorem}
Informally, Theorem \ref{thm:local spacing} states that locally the roots of $A_{n,f_n}$ in the bulk of the limiting measure are very regular, and are asymptotically some linear transformation of the integer lattice $\Z$. Since $\Im G (E+\ii0)=-\pi \frac{\dd\mu_{c,\Sigma}}{\dd E}$ is proportional to the density at $E$, we observe that this lattice of roots becomes denser if the mass at $E$ is higher.	
The constants $C_{E,n}$ and $B_{E,n}$ can be made explicit from our proof, see \eqref{eq:CEn} and \eqref{eq:BEn}.

In the classical example $f_n(x)=\exp(-x^2/2n)$ with corresponding Appell polynomials $A_{n,f_n}(x)=\frac{1}{n^{n/2}}\He_n(\sqrt n x )$, we retrieve the well known Plancherel–Rotach asymptotics \cite{NIST:DLMF}, see Example \ref{ex:Hermite PR check}. We also note that for a fixed function $f$ with regularly varying $n_f$, the scale on which the roots of the Appell polynomials approach perfect spacing is (up to a slowly varying correction) $n^{1/\alpha}$.

\subsection{Results for P\'olya--Schur type operators}
For P\'olya--Schur type operators we will focus on the analogues of Corollary \ref{cor:Free convolution theorem}. We will consider two versions of the result. First, we will consider the special case when $T=f(z\partial_{z})$ for some function $f\in\LP$ with only negative roots. Second, we will consider general $T$ classified by the P\'olya--Schur theorem.\begin{proposition}[P\'olya--Schur Theorem]\label{prop:Polya-Schur}
	Let $T:\mathcal{P}(\C)\rightarrow\mathcal{P}(\C)$ be a linear operator which commutes with the diagonal (in the monomial bases) operator $z\partial_{z}$ and let $\tau_k$ be such that $Tz^k=\tau_kz^k$. Additionally, let $\Phi_{T}$ be the formal power series \begin{equation}\label{eq:Phi_T def}
		\Phi_{T}(z)=\sum_{k=0}^{\infty}\frac{\tau_k}{k!}z^{k}=T\left[\exp( z)\right].
	\end{equation} Then, $T$ preserves real-rootedness if and only if either $\Phi_{T}(z)$ or $\Phi_{T}(-z)$ is an entire function in $\mathcal{LPI}$. Moreover, $T$ maps non-negative roots to real roots if and only if $\Phi_{T}\in\LP$. 
\end{proposition}
	To classify the limiting distribution we will need to briefly introduce the free multiplicative distribution $\boxtimes$. To this end define the $\psi$-transform of a probability measure $\mu$ on $[0,\infty)$ to be the function $\psi_\mu$  on $\C\setminus\{z:z^{-1}\in\supp\mu\}$ given by 
\begin{equation}
	\psi_\mu(z)=\int_{0}^{\infty}\frac{uz}{1-uz}\dd\mu(u).
\end{equation} If $\mu\neq\delta_{0}$, $\psi_\mu$ has an inverse defined on a neighborhood of $(-1+\mu(\{0\}),0)$ and the $S$-transform of $\mu\neq \delta_{0}$ is the function defined by \begin{equation}
	S_{\mu}(z)=\frac{1+z}{z}\psi_\mu^{-1}(z),
\end{equation}  for $z\in(-1+\mu(\{0\}),0)$.  The free multiplicative convolution of two probability measures,  $\mu_{1}\neq \delta_0$ and $\mu_{2}\neq\delta_{0}$, on $[0,\infty)$ is the probability  measure $\mu_{1}\boxtimes\mu_{2}$ with $S$-transform \begin{equation}
	S_{\mu_{1}\boxtimes\mu_{2}}(z)=S_{\mu_{1}}(z)S_{\mu_{2}}(z),
\end{equation} for $z$ on some common interval $(-\eps,0)$ with $S_{\mu_{1}}$ and $S_{\mu_{2}}$ both defined. A probability measure $\mu$ on $[0,\infty)$ is $\boxtimes$-infinitely divisible (or $\boxtimes$-ID) if $\mu^{\boxtimes t}$ exists for all $t>0$. 

\begin{theorem}\label{thm:Polya schur and multiplicative groups}
	Let $\{f_n\}\subset\LP$ be a sequence of functions with only negative roots and which satisfies  Assumption \ref{assump:main assumption}. If $t\geq 0$ and $\{p_n\}_{n\geq 1}$ is a sequence of positively-rooted polynomials indexed by their degree such that $\llbracket p_n\rrbracket\Rightarrow\mu$ as $n\rightarrow\infty$, then \begin{equation}
		\left\llbracket f_{\lfloor tn\rfloor}(tz\partial_{z})p_n(z)\right\rrbracket\Rightarrow \mu\boxtimes \tilde{\mu}_{c,\Sigma}^{\boxtimes t},
	\end{equation} where $\tilde{\mu}_{c,\Sigma}$ is the $\boxtimes$-infinitely divisible probability measure on $[0,\infty)$ with $S$-transform \begin{equation}
		S_{\tilde\mu_{c,\Sigma}}(z)=\exp\left(-R_{\mu_{c,\Sigma}}(z+1) \right),
	\end{equation} for $z\in (-1,0)$. 
\end{theorem}  

The special cases when $f(z)=e^{-z^2+z}$ or  $f(z)=(z+b)^{c}$ were previously studied in \cite{Jalowy-Kabluchko-Marynych2025part2,Kabluchko2026} under the names of multiplicative Hermite and multiplicative Laguerre polynomials, respectively. 

A definition for $\mu\boxtimes\nu$ where only $\nu$ (or $\mu$ by symmetry) is assumed to be supported on $[0,\infty)$ has been available using operator algebras since \cite{Voiculescu1987}, however a definition in terms of $S$-transforms was not available until the recent work of Arizmendi, Hasebe, and Kitagawa \cite{Arizmendi-Hasebe-Kitagawa2026}. For our next result, we will use that $\mu\boxtimes\nu$ exists when at least one is supported on $[0,\infty)$, and encourage the reader to see \cite{Arizmendi-Hasebe-Kitagawa2026} for a precise definition when one of the supports contains both positive and negative numbers. 

\begin{theorem}\label{thm:Polya-Schur full version}
	Let $\{T_{n}\}$ be a series of operators which commute with $z\partial_z$  such that $\Phi_{T_n}\in\mathcal{LP}$ as defined in \eqref{eq:Phi_T def}. If $\{p_n\}_{n\geq 1}$ is a sequence of positively-rooted polynomials indexed by their degree, with uniformly bounded roots such that $\llbracket p_n\rrbracket\Rightarrow\mu$ as $n\rightarrow\infty$, and $\{\Phi_{T_n}\}$ satisfy Assumption \ref{assump:main assumption}, then \begin{equation}
		\left\llbracket T_{n}p_n \right\rrbracket\Rightarrow \mu\boxtimes \mathcal{D}_{-1}\mu_{c,\Sigma}^{-1},
	\end{equation} where $\mu_{c,\Sigma}$ is the $\boxplus$-ID limit of the Appell polynomials $A_{n,\Phi_{T_n}}$ from Theorem \ref{thm:main result, general}.
\end{theorem} The uniform bound on the roots of $p_n$ is only technical, and a result of the current state of the art on understanding the finite free multiplicative convolution. We believe the same result should hold for polynomials $p_n$ with unbounded positive roots. For polynomials $p_n$ with both positive and negative roots, one would need to add the assumption that $\Phi_{T_n}\in\mathcal{LPI}$. 

\begin{remark}\label{rem:optimality}
 The polynomials considered in Theorem \ref{thm:Polya schur and multiplicative groups} have only non-negative roots and the theory of \emph{exponential profiles} \cite[Theorems 2.2 and 2.4]{Jalowy-Kabluchko-Marynych2025} provides necessary and sufficient conditions for the empirical root measures of polynomials with non-negative roots to converge, in terms of convergence of these exponential profiles. For functions $f_n$ satisfying the conditions of Theorem \ref{thm:Polya schur and multiplicative groups}, one of our technical results, Proposition \ref{prop:convergence of R-transforms}, is equivalent to Assumption \ref{assump:main assumption} and can be used to prove convergence of the exponential profile. Thus, Assumption \ref{assump:main assumption} is necessary and sufficient for Theorem \ref{thm:Polya schur and multiplicative groups}. We believe that the same is true for Theorems \ref{thm:main result, general} and \ref{thm:Polya-Schur full version} as well.
\end{remark}

\subsection{Results for rectangular differentiation}  We next consider a class of operators $T$ which do not commute with $\partial_{z}$ nor $z\partial_{z}$. These operators are thus not covered by the P\'olya--Benz nor P\'olya--Schur theorems, but we can still express the effect on zero distributions in terms of free probability. To define these operators we define a differential operator $\mathscr{L}_{\beta}$, which we refer to as \emph{rectangular differentiation}, by 
\begin{equation}\label{eq:M def}
	\mathscr{L}_{\beta}=z\partial_{z}^2+(1+\beta)\partial_{z},
\end{equation} for $\beta>-1$. The ``rectangular'' in the name is explained in \S \ref{sec:finite_free_probability}. We will consider operators of the form $T=f(\mathscr{L}_{\beta})$, where $f\in\mathcal{LPI}$. It follows from Lemma \ref{lem:PS for rect}  that these are all the operators which commute with $\mathscr{L}_{\beta}$ such that $T[z^n]$ has only non-negative roots. 
\begin{theorem}\label{thm:rectangular Appell polynomials}
	Let $\{f_{n}\}\subset\mathcal{LPI}$ and satisfy Assumption \ref{assump:main assumption}. If $\beta\equiv\beta_n$ and $\frac{n}{n+\beta}\rightarrow \lambda\in (0,1]$, then \begin{equation}
		\left\llbracket f_n\left(\frac{\lambda}{n}\mathscr{L}_\beta\right)z^{n}\right\rrbracket\Rightarrow\mathsf{MP}_{1,\lambda}\boxtimes \mu_{c,\Sigma},
	\end{equation} 
	 where $\mu_{c,\Sigma}$ is the $\boxplus$-ID limit of the Appell polynomials $A_{n,g_n}$ from Theorem \ref{thm:main result, general} and $\mathsf{MP}_{1,\lambda}$ is the standard Marchenko--Pastur distribution with density 
	\begin{align*}
		\dd\mathsf{MP}_{1,\lambda}(x)
		=\frac 1 {2\pi\lambda x} \sqrt{\big((1+\sqrt{\lambda})^2-x\big) \big(x-(1-\sqrt{\lambda}\big)^2  }\big)\dd x \text{ on } x\in[(1-\sqrt{\lambda})^2, (1+\sqrt{\lambda})^2].
	\end{align*}
\end{theorem} 
 The case of $f_n(z)=\exp(-cz)$ has been studied in \cite[Theorem 5.14, Case I]{Jalowy-Kabluchko-Marynych2025part2}. Theorem \ref{thm:rectangular Appell polynomials} describes the zero distribution of these operators applied to $z^n$, and again from \cite[Theorems 2.2 and 2.4]{Jalowy-Kabluchko-Marynych2025} one can conclude that Assumption \ref{assump:main assumption} is both necessary and sufficient for the result to hold. It should be possible using these exponential profiles, identities from free probability, and a more involved version of Lemma \ref{lem:PS for rect} below to extend this result to general $p_n$ with non-negative roots, similar to Corollary \ref{cor:Free convolution theorem} or Theorem \ref{thm:Polya schur and multiplicative groups}. However, this would require several pages of computations, and we believe this is better left as specific case of a much more general result which can be obtained by a better understanding of the \emph{finite free rectangular convolution} $\boxplus_{n}^{\beta}$ of \cite{Gribinski2024} for non-integer $\beta$. We leave this to future work.  
 
 \subsection{Outline of the proofs}\label{sec:outline_of_proof} The proof of Theorem \ref{thm:main result, general} will follow from the representation of any Appell polynomial in terms of $f$, given by \begin{equation}\label{eq:Appell contour}
 	\begin{aligned}
 		A_{n+d,f}(z)&=f(\partial_{z})z^{n+d}\\
 		&=f(\partial_{z})\frac{(n+d)!}{2\pi\ii}\oint_{\Gamma}e^{xz}\frac{\dd x}{x^{n+d+1}}\\
 		&=\frac{(n+d)!}{2\pi\ii}\oint_{\Gamma} f(x)e^{xz}\frac{\dd x}{x^{n+d+1}},
 	\end{aligned} 
 \end{equation} where $\Gamma$ is any simple closed contour around the origin. The idea behind the saddle point method is to choose the contour $\Gamma$ such that $\left|f(x)e^{xz}/x^{n}\right|$ is maximized on $\Gamma$ at a single point $x_n^*$  as $n\rightarrow\infty$, and that the integral is asymptotic to the contribution in a neighborhood of this point. Holomorphic functions do not have local maximums, but the point $x_n^*$ should still be at a critical point of $f(x)e^{zx}/x^{n}$. Solving for this critical point we are looking for solutions to \begin{equation*}
 \frac{f(x)e^{xz}}{x^n}n\left(\frac{z}{n}+\frac{1}{n}\frac{f'(x)}{f(x)}-\frac{1}{x}\right)=0.
 \end{equation*} This equation likely has an infinite number of solutions.  If $f\in\LP$, then   $\frac{1}{n}\frac{f'(x)}{f(x)}$ is the $R$-transform of the $\boxplus$-ID measure $\mu_{f,n}=\mu_{f}^{\boxplus \frac{1} {n}}$, see \cite[\S  2.1]{campbell2025freeinfinitedivisibilityfractional}.  By \eqref{eq:Cauchy R relationship}, one critical point is given by \begin{equation*}
 x_{n}^*(z)=G_{\mu_{f,n}}\left(\frac{z}{n}\right),
 \end{equation*}
 which is separated from all the critical points close to $\R$.
 To justify that this saddle point is dominating the others on some carefully chosen contour $\Gamma$, and that this fact continues to hold for $f\in\LP_M$, is one of the technical steps of the proof. Here, we will make use of Assumption \ref{assump:main assumption}, the unique definition \eqref{eq:Cauchy R relationship} of the $R$-transform in the relevant domain, and some harmonic analysis involving the limiting saddle point equation as well as topological arguments on its sublevel sets, see Figure \ref{fig:heatmap} below. 
 
  If we assume that the saddle point argument holds, then \begin{equation*}
 A_{n+d,f}(z)\approx\frac{(n+d)!}{2\pi\ii}\int_{x\approx x_{n}^*(z)}  f(x)e^{xz}\frac{\dd x}{x^{n+d+1}}. 
 \end{equation*}
  Then, applying this with $d=-1,0$ we arrive at the following heuristic approximation of the Cauchy transform $G_n$ of $\llbracket A_{n,f}\rrbracket$\begin{equation*}
 G_{n}(z)=\frac{1}{n}\frac{A_{n,f}'(z)}{A_{n,f}(z)}=\frac{A_{n-1,f}(z)}{A_{n,f}(z)}\approx \frac{\frac{(n-1)!}{2\pi\ii}\int_{x\approx x_{n}^*(z)}  f(x)e^{xz}\frac{\dd x}{x^{n}}}{\frac{n!}{2\pi\ii}\int_{x\approx x_{n}^*(z)}  f(x)e^{xz}\frac{\dd x}{x^{n+1}}} \approx n x_{n}^*(z)=nG_{\mu_{f,n}}\left(\frac{z}{n}\right). 
 \end{equation*} From here one would want to normalize the polynomials in $n$ such that the right-hand side converges to the Cauchy transform of some probability measure. This argument is made rigorous in Section \ref{sec:proof}. 
  Rescaling the full asymptotic of the saddle point method then allows to obtain the local spacing of Theorem \ref{thm:local spacing}.
 
 The proofs of our other results follow from our general main Theorem \ref{thm:main result, general}, or technical results proven along the way, in combination with identities in free probability.  For instance, the proof of Theorem \ref{thm:Polya schur and multiplicative groups} does not need another saddle point problem, but instead uses exponential profiles of the respective coefficients obtained from the logarithmic potential of $\Sigma$.

\section{Expanded background and related work}\label{sec:Background}  In this section we present more technical background needed in the proofs and how our work relates to the existing literature.

\subsection{Free probability}  We now discuss free probability in more detail, however we still encourage the interested reader to see \cite{Mingo-Speicher2017,Nica-Speicher2006} and other texts for a more thorough introduction to free probability. 

	There exists a bijection between $\boxplus$-ID measures and classical infinitely divisible measures, known as the Bercovici--Pata bijection after \cite{Bercovici-Pata1999}. We omit the details, but the Bercovici--Pata bijection is in fact a larger bijection between probability measures which preserves the domains of attraction to infinitely-divisible distributions. So the theory of classical infinite-divisibility can be immediately translated in many cases to free probability. It is worth pointing out that in \cite{Bercovici-Pata1999,Bercovici-Voiculescu1993,Bercovici-Voiculescu1995,Bercovici-Wang-Zhong2018} and other references we will use, it is more common to work with the \emph{Voiculescu} transform \begin{equation}\label{eq:Voic to R}
  	\phi_{\mu}(z)=R_{\mu}\left(\frac{1}{z}\right),
  \end{equation} defined on the upper-half plane and the $F$-transform \begin{equation}\label{eq:F to G}
  F_{\mu}(z)=\frac{1}{G_{\mu}(z)}.
  \end{equation} For our purposes, it is more convenient to work with the $R$-transform and Cauchy transform, and the results we need  to apply from these other papers can be immediately translated to our setting using \eqref{eq:Voic to R} and \eqref{eq:F to G}. 
  
  Free L\'evy processes were introduced by Biane \cite{Biane1998} in the study of processes $(X_t)_{t>0}$ with free increments. Notably these processes do not necessarily have both stationary increments and stationary Markov transition functions, leading to two different potential definitions of what could naturally be called a free L\'evy process. When we refer to the distribution of a free L\'evy process we always mean stationary increments with distributions of the form $\mu_{t}=\mu^{\boxplus t}$ for some $\boxplus$-ID $\mu$. 
  
  Free L\'evy processes are a helpful way of understanding why the limit in Corollary \ref{cor:fixed f and stable laws} must be free stable. For $f\in\LP$, let $\mu_{f}$ denote the $\boxplus$-ID measure with $R$-transform  $R_{\mu_f}(u)=-\frac{f'(u)}{f(u)}$. As discussed  in Section \ref{sec:outline_of_proof},  
   the measure $\mu_{n}$ with $R$-transform $R_{\mu_n}(u)=R_{\mu_f}(nu)$ is our $\boxplus$-ID approximation of $\llbracket A_{n,f}\rrbracket$. It additionally follows that $\mu_{n}=\mathcal{D}_{n}\mu_{f}^{\boxplus{1}/{n}}$. If $\mu_{n}$ is a good approximation of $\llbracket A_{n,f}\rrbracket$, then the limit in Corollary \ref{cor:fixed f and stable laws} must be the small time limit of a shifted and re-scaled free L\'evy process. The only possible such limits must be stable for both classical and free L\'evy processes \cite{Arizmendi-Hasebe2018,Maller-Mason2008}, and the conditions of Corollary \ref{cor:fixed f and stable laws} are known to be necessary and sufficient for this limit to exist.

\subsection{Finite free probability}\label{sec:finite_free_probability}  Finite free probability emerged out the celebrated work of Marcus, Spielman, and Srivastava \cite{Marcus-Spielman-Srivastava2015-1,Marcus-Spielman-Srivastava2015-2,Marcus-Spielman-Srivastava2022} on interlacing polynomials. It has become a successful tool for studying how differential operators affect polynomial roots. We encourage the interested reader to see \cite{Arizmendi-Perales2018,Marcus-Spielman-Srivastava2022, Jalowy-Kabluchko-Marynych2025} for more on finite free probability. 

The key tools for extending the actions of differential operators from monomials to more general polynomials are the finite free additive $\boxplus_{n}$ and multiplicative $\boxtimes_{n}$ convolutions on the space of polynomials. These have many equivalent definitions, see \cite[\S 1.1]{Marcus-Spielman-Srivastava2022}, even though the original definition already appeared in \cite{walsh} and \cite{szegoe_bemerkungen_grace}.  \begin{definition}\label{def:finite boxplus}
	Let $p_n$ and $q_n$ be degree $n$ monic polynomials with corresponding polynomials  $P_n$ and $Q_{n}$ of degree $n$ such that $p_n(z)=P_n(\partial_{z})z^{n}$ and $q_{n}(z)=Q_n(\partial_{z})z^n$, respectively.  
	Then, the finite free additive convolution  of $p_n$ and $q_n$ is the unique degree $n$ polynomial $p_n\boxplus_{n}q_n$ such that \begin{equation}\label{eq:def:finite boxplus}
		p_n\boxplus_{n}q_{n}(z)=Q_n(\partial_{z})P_n(\partial_{z})z^n=Q_{n}(\partial_{z})p_n(z)=P_{n}(\partial_z)q_n(z).
	\end{equation}
	Equivalently, if $p_n(z)=\sum_{k=0}^{n}(-1)^{k}e_{k}(p_n)\binom{n}{k}z^{n-k}$ and $q_n(z)=\sum_{k=0}^{n}(-1)^{k}e_{k}(q_n)\binom{n}{k}z^{n-k}$, then \begin{equation*}
		p_n\boxplus_{n}q_n(z)=\sum_{k=0}^{n}z^{n-k}(-1)^{k}\binom{n}{k}\Bigg(\sum_{j=0}^k\binom k j  e_{j}(p_n)e_{k-j}(q_n)\Bigg) . 
	\end{equation*}
\end{definition}

 For our work, \eqref{eq:def:finite boxplus} will turn out to be most convenient. The polynomials $P_n,Q_n$ are unique if their degree $n$ is fixed, but the same holds for any other formal power series that is equivalent modulo $z^n$.
If $p_n,q_n$ are characteristic polynomials of Hermitian $n\times n$ matrices $A,B$, one can also represent $p_n\boxplus q_n$ as the characteristic polynomial of $A+UBU^*$ for Haar-unitary $U$. This explains the name of the \emph{finite} analogue of the free convolution.

\begin{definition}\label{def:finite boxtimes}
		Let $p_n$ and $q_n$ be degree $n$ monic polynomials with corresponding unique polynomials  $P_n$ and $Q_{n}$ such that $p_n(z)=P_n(z\partial_{z})(z-1)^{n}$ and $q_{n}(z)=Q_n(\partial_{z})(z-1)^n$, respectively.
	Then, the finite free multiplicative convolution  of $p_n$ and $q_n$ is the unique degree $n$ polynomial $p_n\boxtimes_{n}q_n$ such that \begin{equation*}
		p_n\boxtimes_{n}q_{n}(z)=Q_n(z\partial_{z})P_n(z\partial_{z})(z-1)^n.
	\end{equation*} Equivalently, if $p_n(z)=\sum_{k=0}^{n}(-1)^{k}e_{k}(p_n)\binom{n}{k}z^{n-k}$ and $q_n(z)=\sum_{k=0}^{n}(-1)^{k}e_{k}(q_n)\binom{n}{k}z^{n-k}$, then \begin{equation*}
	p_n\boxtimes_{n}q_n(z)=\sum_{k=0}^{n}(-1)^{k}e_{k}(p_n)e_{k}(q_n)\binom{n}{k}z^{n-k}. 
	\end{equation*}
\end{definition} 

These convolutions have many interesting properties, including the fact that $p_n\mapsto p_n\boxplus_n q_n$ preserves real roots for any real-rooted $q_n$, and $p_n\mapsto p_n\boxtimes_n q_n$ preserves real roots for any non-negatively-rooted $q_n$, see \cite{Marcus-Spielman-Srivastava2022}. We will use that these induce convolutions on empirical root measures, which ``converge'' to the free convolutions of \S \ref{subsec:free prob} in the large degree limit, see \cite{Marcus2021,Arizmendi-Perales2018,Arizmendi-GarzaVargas-Perales2023,Arizmendi-Fujie-Perales-Ueda2024,Jalowy-Kabluchko-Marynych2025,Fujie2026} for various settings and proofs.

\begin{lemma}[Theorem 1.3 in \cite{Fujie2026}]\label{lem: convolution convergence}
	Let $\{p_{n}\}_{n\geq 1}$ and $\{q_n\}_{n\geq 1}$ be sequences of real rooted polynomials indexed by their degrees such that $\llbracket p_n\rrbracket\Rightarrow\mu$ and $\llbracket q_n\rrbracket\Rightarrow \nu$ for some probability measures $\mu$ and $\nu$. Then, \begin{equation}
		\llbracket p_n\boxplus_{n} q_{n}\rrbracket\Rightarrow\mu\boxplus\nu.
	\end{equation}
\end{lemma}   The multiplicative analogue of Lemma \ref{lem: convolution convergence} is available from \cite{Arizmendi-Fujie-Perales-Ueda2024}.  \begin{lemma}[Proposition 10.1 in \cite{Arizmendi-Fujie-Perales-Ueda2024}]\label{lem: mult convolution convergence}
Let $\{p_{n}\}_{n\geq 1}$ and $\{q_n\}_{n\geq 1}$ be sequences of positively rooted polynomials indexed by their degrees such that $\llbracket p_n\rrbracket\Rightarrow\mu$ and $\llbracket q_n\rrbracket\Rightarrow \nu$ for some probability measures $\mu$ and $\nu$. Then, \begin{equation}
	\llbracket p_n\boxtimes_{n} q_{n}\rrbracket\Rightarrow\mu\boxtimes\nu.
\end{equation}
\end{lemma} 
Another version of this result is due to Fujie \cite{Fujie2026}, where one can relax the assumptions on $q_n$ at the cost of added assumptions on $p_n$.  \begin{lemma}[Theorem 4.5 in \cite{Fujie2026}]\label{lem: mult convolution convergence Fujie}
Let $\{p_{n}\}_{n\geq 1}$ and $\{q_n\}_{n\geq 1}$ be sequences of polynomial indexed by their degrees such $p_n$ has non-negative uniformly bounded roots, $q_n$ has real roots, and $\llbracket p_n\rrbracket\Rightarrow\mu$, and $\llbracket q_n\rrbracket\Rightarrow \nu$ for some probability measures $\mu$ and $\nu$. Then, \begin{equation}
	\llbracket p_n\boxtimes_{n} q_{n}\rrbracket\Rightarrow\mu\boxtimes\nu.
\end{equation}
\end{lemma} Marcus \cite{Marcus2021} also defined the \emph{finite $R$-transform} of $p_n$ as the formal power series \begin{equation}
R_{p_n}(u)=-\frac{P_n'(nu)}{P_{n}(nu)}\mod u^{n}.
\end{equation}
By Definition \ref{def:finite boxplus}, the finite free convolution $	p_n\boxplus_{n}q_{n}$ corresponds to a multiplication $P_nQ_n$, and hence $R_{	p_n\boxplus_{n}q_{n }}=R_{p_n}+R_{q_n}$ by the product rule, as one would expect from a finite version of the $R$-transform. When $p_n(z)=f(\partial_z)z^{n}=A_{n,f}(z)$ for $f\in\LP$, one could remove the truncation $\mod u^{n}$ and consider $R_{p_n}$ as an analytic function on $\C_-$. This is where we get the approximation $\mu_{n}$ of $\llbracket A_{n,f}\rrbracket$.

The operator $\mathscr{L}_{\beta}$ defined in \eqref{eq:M def} can also be used to define a finite convolution, referred to as the finite free rectangular convolution $\boxplus_{n}^{\beta}$. The operation $\boxplus_{n}^{\beta}$ was first defined by Gribinski \cite{Gribinski2024} using explicit formulas for the coefficients. It was pointed out explicitly in \cite{campbell2025freeinfinitedivisibilityfractional} that this is equivalent to Definition \ref{def:finite boxplus}, with $\mathscr{L}_{\beta}$ replacing $\partial_{z}$, though this fact is implicit in one of the proofs of Cuenca \cite{Cuenca24}. 
The ``rectangular'' in the name of $\mathscr{L}_{\beta}$ and $\boxplus_{n}^{\beta}$ stems from characterizations in terms of random rectangular matrices when $\beta\in\Z$, see \cite{Gribinski2024} and \cite[Lemma 1.10]{campbell2025freeinfinitedivisibilityfractional}. It appears open whether $\boxplus_{n}^{\beta}$ preserves positive roots for non-integer $\beta$. The analogous result to Lemmas \ref{lem: convolution convergence}, \ref{lem: mult convolution convergence}, and \ref{lem: mult convolution convergence Fujie} remains open as well, except for the case when the roots are uniformly bounded, which is due to \cite{Cuenca24}. If these problems were resolved, it would be immediate to extend Theorem \ref{thm:rectangular Appell polynomials} to more general $f(\mathscr{L}_{\beta})p_n$. 

\subsection{Differential operators on polynomials} Finite free probability has lead to significant progress on understanding the effects of differential operators acting on polynomial roots and how they connect to free probability. Two operators which have received substantial attention in free probability are $\partial_{z}^{tn}$ and $e^{-\frac{t}{2n} \partial_{z}^2 }$, where $n$ is the degree of the polynomial on which they act.  Let us discuss the previous results in more detail.


 For \emph{repeated differentiation}, the corresponding Appell polynomial is trivially a rescaled monomial again, whose limiting zero distribution is $\delta_0$. Historically, the most fundamental result is the Gauss--Lucas theorem~\cite[Theorem 6.1]{Marden}, locating the roots of the derivative inside the convex hull of the roots of the original polynomial. In particular, a finite number of differentiations $\partial_z^kp_n$ does not change the limiting root distribution $\llbracket p_n\rrbracket\Rightarrow \mu$, see \cite{Byun-Lee-Reddy2022,totik,Michelen-Vu2024,Angst-Malicet-Poly2024,angst2026convergence,michelen2024almost}. Interestingly, if the order of differentiation $\partial_z^{tn}$ is the same as the degree of the polynomial, then a new limiting root distribution emerges, which first have been described via non-local PDE's in
\cite{Steinerberger2019}. It was then connected to free projections and free convolution powers in \cite{Shlaykhtenko-Tao2020,Steinerberger2020,martinez2024flow,kiselev2022flow}. Finally, the weak convergence $\llbracket \partial_z^{tn}p_n\rrbracket\Rightarrow \mathcal D_{1-t}\mu^{\frac{1}{1-t}}$ was proven in \cite{Hoskins-Kabluchko2021}, see also Corollary \ref{cor:f with a root at 0} and \cite{Arizmendi-GarzaVargas-Perales2023}. By now, generalizations include $(z^a\partial_z^b)^{tn}(z-1)^n$ with some $\boxtimes$-ID limiting root distribution, see \cite{Jalowy-Kabluchko-Marynych2025}, and the polar differentiation with limiting distribution described via the M\"obius transform. Moreover, local descriptions of roots under repeated differentiation are available from an approximation $\partial_z^{n-d} p_n(\tfrac{z}{\sqrt n})\approx \mathrm{He}_d(z)$ for fixed $d$ and as $n\to\infty$, see \cite{Hoskins-Steinerberger2022, Campbell-ORourke-Renfrew2024even,Arizmendi-Campbell-Fujie2025,Gorin-Klepttsyn2020universal}. Notably, the framework has been generalized to some complex rooted polynomials $\partial_z^{tn} p_n$, see \cite{Campbell-ORourke-Renfrew2024,BHS24,Hall-Ho-Jalowy-Kabluchko2023Repeat,GNV25}, but still the full universality of complex rooted polynomials remains open.

For the (holomorphic, backwards) \emph{heat flow}, we saw in Corollary \ref{cor:fixed f and stable laws} that the corresponding Appell polynomials are rescaled probabilist's Hermite polynomials $
\exp(-\frac{\sigma^2}{2n}\partial_z^2 )z^{n} = \He_n(\sqrt n z/\sigma)$ whose limiting root distribution is the semicircle law $\mathsf{sc_{\sigma}}$ of variance $\sigma^2$. More generally, it has been shown in \cite{Kabluchko2022leeyang,VW22,Jalowy-Kabluchko-Marynych2025part2} that if $p_n$ is real rooted with $\llbracket p_n\rrbracket\Rightarrow \mu$, then $\llbracket\exp(-\frac{\sigma^2}{2n}\partial_z^2 )p_n\rrbracket\Rightarrow \mu\boxplus\mathsf{sc}_\sigma$ is the free heat semigroup. Also here, specific settings of complex rooted polynomials have been studied in \cite{Hall-Ho-Jalowy-Kabluchko2023,Hall-Ho-Jalowy-Kabluchko2023heat,Hall-Ho2022heat,Hofert}. Related exponentiated differential operators of the form $\exp\big(-(\frac{z\partial_z}{n})^k\partial_z\big)z^n$ have been studied in \cite{Jalowy-Kabluchko-Marynych2025part2} whose limiting root distributions is some specific $\boxtimes$-ID distribution, as well as $f\big(\frac{\partial_z}{n}\big)^nz^n=(\mathcal D_{1/n} A_{n,f})^{\boxplus_n n}(z)$ in \cite{campbell2025freeinfinitedivisibilityfractional} whose limiting root distribution is some specific $\boxplus$-ID distribution.

 Therefore, $\partial_{z}^{tn}p(z)$ and $e^{-\frac{t}{2n} \partial_{z}^2 }p(z)$ provide finite free analogues of \emph{free projection} and \emph{free Brownian motion}, respectively. Since both $z^{\lfloor tn\rfloor}$ and $e^{-\frac{t}{2n}{z}^2 }$ are in $\LP$, Corollary \ref{cor:f with a root at 0} unites these previous result and generalizes them to the largest possible class of power series in $\partial_{z}$ which preserve real roots. 
Many of the previous results on roots of real-rooted polynomials under differential flows primarily rely on combinatorial tools from finite free probability involving moments of the roots, hence they are intrinsically restricted to (1) real rooted polynomials, (2) usually with bounded support and (3) global results on the limiting root distribution.

\begin{remark}
Let us comment on a PDE-perspective of the differentiation flow $\llbracket f_{\lfloor tn\rfloor}\left(t\partial_{z}\right)p_n(z)\rrbracket\Rightarrow \mu\boxplus\mu_{c,\Sigma}^{\boxplus t}$ of Corollary \ref{cor:Free convolution theorem}. For the heat flow of time $t>0$ and at any initial distribution $\mu$ on $\R$, it goes back to Voiculescu \cite[\S4.7]{voiculescu1990noncommutative} that the Cauchy transform $G_t$ of the distribution $\mu\boxplus\mathsf{sc}_t$ satisfies the Burger's equation 
\begin{align}\label{eq:Burgers} \partial_t G_t(z)=-  G_t(z)\partial_zG_t(z)\text{ on } z\in\C_+\text{ with  } G_0(z)=G_\mu(z),
\end{align}
 see also \cite{Biane,Voiculescu-Dykema-Nica1992}. More generally, for a free convolution flow $\mu\boxplus\mu_{c,\Sigma}^{\boxplus t}$, its $R$-transform satisfies $R_t=R_\mu+tR_{\mu_{c,\Sigma}}$ by Definition \ref{def:free conv}. Hence, the inverse Cauchy transform is given by $G_t^{-1}=G_\mu^{-1}+tR_{\mu_{c,\Sigma}}$. Differentiating yields $0=\partial_t\big( G_t(G_\mu^{-1}(z)+tR_{\mu_{c,\Sigma}}(z))\big)$, which yields
 \begin{align}
\partial_t G_t(z)=-  R_{\mu_{c,\Sigma}}(G_t(z))\partial_zG_t(z)\text{ on } z\in\C_+\text{ with  } G_0(z)=G_\mu(z),
 \end{align}
 generalizing the Burger's equation \eqref{eq:Burgers}. 
\end{remark}

\subsection{The P\'olya--Schur program}\label{sec:polyaschur}  The P\'olya--Schur program \cite{Borcea-Branden2009Ann,Borcea-Branden2009Inv,Craven-Csordas1994,Benz1935,Polya1913,Schur--Polya1914,Laguerre1883} is a series of results aiming to solve the following problem. \begin{problem}\label{prob:PS prob}
	For $\Omega\subset\C$, let $\mathcal{P}(\C,\Omega)\subset\mathcal{P}(\C)$ denote all univariate polynomials with all their zeros in $\Omega$. Classify the set of all linear operators $T$ on $\mathcal{P}(\C)$ such that $T(\mathcal{P}(\C,\Omega))\subseteq \mathcal{P}(\C,\Omega)$. 
\end{problem} A natural and important version of this problem is the case of operators preserving real roots. Special cases date back to the works of Laguerre, P\'olya, Schur, and Benz  \cite{Benz1935,Polya1913,Schur--Polya1914,Laguerre1883} among others. The full classification was given by Borcea and Br\"and\'en \cite{Borcea-Branden2009Ann}, nearly a century after these classical results.

One can go beyond asking not only \emph{where} the zeros are distributed, but also \emph{how} they are distributed in the large degree limit. 
\begin{question}\label{ques:how are roots distributed}
	For a sequence of polynomials $\{p_n\}\subset\mathcal{P}(\C,\Omega)$ of growing degree and an operator $T$ preserving roots in $\Omega$, how are the roots of $Tp_n$ distributed as $n\rightarrow\infty$?
\end{question}   Our results answer this question for the operators classified by the classical works of P\'olya, Schur, and Benz.  
Laguerre provided a earlier family of operators which preserve real roots, which can be seen as a special case of the P\'olya--Schur theorem.
 \begin{lemma}[Laguerre's Theorem]\label{lem:Lag theorem}
	Let $f$ be a function in the Laguerre--P\'olya such that $f$ has only negative roots. If all the roots of $p$ are real, then all the roots of $f(z\partial_{z})p(z)$ are real. 
\end{lemma}

Problem \ref{prob:PS prob} remains open in the case when $\Omega=\R_+$, and we were unable to find a version of the following lemma in the literature.
\begin{lemma}[Weak P\'olya--Schur Theorem for $\mathscr{L}_{\beta}$]\label{lem:PS for rect}
	Let $T$ be a linear operator that commutes with $\mathscr{L}_{\beta}$. $T[z^n]$ has only non-negative roots for every positive integer $n$  if and only if  \begin{equation}
		T=f(\mathscr{L}_{\beta}),
	\end{equation} for some function $f\in\mathcal{LPI}$.
\end{lemma} The proof of this lemma follows much in the same way as the P\'olya--Schur theorem, and we include it in Appendix \ref{appendix}. We were also unable to find a proof of either the P\'olya--Schur or P\'olya--Benz theorems in any open source reference in English. The ``only if'' direction of our proof of Lemma \ref{lem:PS for rect} can also serve as a proof of the P\'olya--Benz theorem with $\mathscr{L}_{\beta}$ replaced by $\partial_{z}$ plus some obvious adjustments.

\section{The proof of Theorem \ref{thm:main result, general}}\label{sec:proof} To prove Theorem \ref{thm:main result, general} we proceed by proving asymptotics for the Appell polynomials as $n\rightarrow\infty$ in Theorem \ref{thm:Appell Planc-Roa asymp} below. Before stating this theorem we establish some notation and set up the main ideas of the proofs.


Applying \eqref{eq:Appell contour} to $A_{n+d,f_n}$ for some fixed $d\in\Z$, we set up a saddle point argument, where the exact contour may change from line to line, \begin{equation}\label{eq:clean saddle point set up}
\begin{aligned}
	A_{n+d,f_n}(z)&=\frac{(n+d)!}{2\pi\ii}\oint_{\Gamma}e^{xz}f_{n}(x)\frac{\dd x}{x^{n+d+1}}\\
	&=\frac{(n+d)!}{2\pi\ii}\oint_{\Gamma}\exp\left(xz+\log f_{n}(x)-n\log x\right)\frac{\dd x}{x^{d+1}}\\
	&=\frac{(n+d)!}{2\pi\ii n^{n+d}}\oint_{\Gamma}\exp\left(n\left[uz+\frac{1}{n}\log f_{n}(nu)-n\log u\right]\right)\frac{\dd u}{u^{d+1}}.
\end{aligned}
\end{equation} We define the function \begin{equation}\label{eq:saddle point function}
h_n(u):=zu+\frac{1}{n}\log f_{n}(nu)-\log u,
\end{equation} where $\log f_{n}(nu)$ is an analytic function on the simply connected domain $\{u:\Im u<-\frac{M_n}{n}\}$ with real part $\log|f_{n}(nu)|$ and such that $e^{\log f_{n}(nu)}=f_{n}(nu)$. This uniquely defines $h_n$ on this domain up to an imaginary constant where the exact choice has no effect on the value of \begin{equation}
\int_{\Gamma'}\exp(nh_n(u))\frac{\dd u}{u^{d+1}},
\end{equation} for any curve $\Gamma'\subset\{u: \Im u<-\frac{M_n}{n} \}$. $h_n$ can also be defined in an analogous way as a holomorphic function on any simply connected domain which does not contain $0$ nor a root of $f_{n}(nu)$.   
 
 In the saddle point method one chooses the contour $\Gamma$ to pass through a critical point $u_{n}^*$ of $h_n$ such that $\Re h_n(u)<\Re h_{n}(u_{n}^*)$ for all other $u\in\Gamma$ and the integral can be approximated using a Taylor expansion of $h_n$ about $u_{n}^*$, and the contribution from the remainder of the contour is asymptotically negligible. See \cite{OSullivan2019} for more details on the saddle point method. It is worth noting that $h_n$ is not holomorphic on a neighborhood of the entire contour $\Gamma$, which is a typical assumption for the most straightforward rigorous statements of the saddle point method, for example in \cite[Theorem 1.2]{OSullivan2019}. However, the function we are truly integrating in \eqref{eq:clean saddle point set up} is holomorphic on $\C\setminus\{0\}$, and we can define $h_{n}$ piece-wise on simply connected neighborhoods of $\Gamma$ which cover $\Gamma$ minus possibly a finite number of points at zeros of $f_n(nu)$, such that $\Re h_{n}(u)$ is continuous on $\Gamma$, minus possibly a finite number of points where $\Re h_n(u)$ tends to $-\infty$, and \begin{equation}
 	\oint_{\Gamma}\exp\left(nh_n(u)\right)\frac{\dd u}{u^{d+1}}=\oint_{\Gamma}e^{nuz}f_{n}(nu)\frac{\dd u}{u^{n+d+1}}.
 \end{equation} We will elaborate more in the proof of Theorem \ref{thm:Appell Planc-Roa asymp} and for now we proceed to study $h_n$ on the domain $\{u:\Im u<-\frac{M_n}{n} \}$.

We are interested in solutions to the saddle point equation \begin{equation}\label{eq:clean saddle point equation}
	h_{n}'(u)=z+\frac{f_{n}'(nu)}{f_{n}(nu)}-\frac{1}{u}=0.
\end{equation} This equation may have an infinite number of solutions in $\C$, particularly near the roots of $f_n(nu)$ for large $n$, and it is not immediately obvious which of these is actually relevant to our analysis. To help identify this distinguished point  we define the meromorphic function \begin{equation}\label{eq:clean R transform}
R_n(u)=-\frac{f_{n}'(nu)}{f_n(nu)}.
\end{equation} If $f_{n}\in\LP$, then $R_{n}$ is in fact the $R$-transform of a $\boxplus$-ID probability measure $\mu_{n}$ and \begin{equation}
G_{\mu_{n}}\left(\frac{1}{u}+R_n(u)\right)=u,
\end{equation} for all $u\in\Omega=\{w\in\C_{-}: \Im \left(\frac{1}{w}+R_{n}(w)\right)>0\}$, where $G_{\mu_{n}}$ is the Cauchy transform of $\mu_{n}$. 

It is now reasonable to guess that $R_n$ is close to an $R$-transform, and that solutions $u_n^*(z)$ to the saddle point equation $R_n(u)+1/u=z$ are close to a Cauchy transform. Let us record these observations in the following proposition and lemma. 
 \begin{proposition}\label{prop:convergence of R-transforms}
	Let $\{f_{n}\}_{n\geq 1}$ satisfy Assumption \ref{assump:main assumption}. 
	\begin{enumerate}
		\item Let $R_{\mu_{c,\Sigma}}$ be the $R$-transform of $\mu_{c,\Sigma}$, then 
		\begin{equation}
			\lim\limits_{n\rightarrow\infty}-\frac{f_{n}'(nu)}{f_{n}(nu)}= R_{\mu_{c,\Sigma}}(u),
		\end{equation} 
		uniformly on compact subsets of $\C_{-}$.
		
		\item Define $\left[ \log(1-xw)\frac{x^2+1}{x^2}+\frac{w}{x}\right]\Big|_{x=0}:=-\frac{w^2}{2}$ by continuous extension, then		
		\begin{equation}\label{eq:log of f_n convergence}
			\lim\limits_{n\rightarrow\infty}\frac{1}{n}\log f_{n}(nu)-\frac{1}{n}\log f_{n}(-n\ii)=-c(u+\ii)+\int_{\R}\log\left(\frac{1-ux}{1+\ii x}\right)\frac{x^2+1}{x^2}+\frac{u+\ii}{x}\dd\Sigma(x),
		\end{equation}  uniformly on compact subsets of $\C_{-}$.
		\item For any $M\in\R$ define $\log_{M}|z|:=\max(\log|z|,-M)$, then \begin{align}\label{eq:cutoff log of f_n convergence}
			\lim\limits_{n\rightarrow\infty} &\frac{1}{n}\log_{M}| f_{n}(nu)|-\frac{1}{n}\log_{M}| f_{n}(-n\ii)|\\
			&=\Re\left(-c(u+\ii)\right)+\int_{\R}\left[\log_{M}\left|{1-ux}\right|-\log_{M}\left|{1+\ii x}\right|\right]\frac{x^2+1}{x^2} +\Re\frac{u+\ii}{x}\dd\Sigma(x),\nonumber
		\end{align} uniformly on compact subsets of $\C_{-}\cup\R\setminus\{0\}$.
	\end{enumerate} 
\end{proposition} 

The choice of $\ii$ in $\frac{1}{n}\log f_{n}(-n\ii)$ in \eqref{eq:log of f_n convergence} was arbitrary. The purpose of this term is to guarantee an antiderivative of $R_n$ which is not blowing up in $n$. 
\begin{remark}
	 As discussed in Remark \ref{rem:optimality}, the asymptotic \eqref{eq:log of f_n convergence} is equivalent to convergence of the exponential profiles of the polynomials considered in Theorem \ref{thm:Polya schur and multiplicative groups}, as we will see during its proof. Additionally, recalling that $\Delta \log|z|=2\pi\delta_{0}$ in the distributional sense, it is not hard to show \eqref{eq:log of f_n convergence} is equivalent to Assumption \ref{assump:main assumption} when $f_{n}\in\LP$.
\end{remark}

The following lemma captures the properties of solutions to \eqref{eq:clean saddle point equation} we will need.

\begin{lemma}\label{lem:clean lem_unique_saddlepoint}
\begin{enumerate}
	\item Let $\mu$ be any freely infinitely divisible distribution. Then, for any $z\in\C_{+}\cup\{z\in\R: -\infty<\Im G_{\mu}(z+\ii0)<0\}$
	 \begin{equation}\label{eq:general R transform version saddle point.}
	z-\frac{1}{u}-R_{\mu}(u)=0
\end{equation} has a unique solution $G_{\mu}(z)$ in $\C_-$. Hence, if $f_n\in\LP$ and $z\in\C_{+}\cup\{z\in\R:-\infty< \Im G_{\mu_{c,\Sigma}}(z+\ii0)<0 \}$, then \eqref{eq:clean saddle point equation} has a unique solution in $\C_{-}$ given by $u_n^*(z)=G_{\mu_n}(z)$.

\item If $\{f_{n}\}$ satisfy Assumption \ref{assump:main assumption}, then for any $\eps>0$ \eqref{eq:clean saddle point equation} has at most one solution in $\{u\in\C:\Im u<-\eps\}$ for $n$ sufficiently large. Additionally, if $\eps>0$ is sufficiently small, then this unique solution $u_n^*(z)$ exists. 

\item Let $\{f_n\}$ satisfy Assumption \ref{assump:main assumption} and let $B\subset\C_{+}\cup\{z\in\R: -\infty<\Im G_{\mu_{c,\Sigma}}(z+\ii0)<0 \}$ be a compact subset. Let $\eps>0$ be small enough such that \eqref{eq:clean saddle point equation} has a unique solution $u=u_n^*(z)\in\{u\in\C:\Im u<-\eps\}$ for all $z\in B$. Then, $u_n^*$ is analytic on a neighborhood of $B$ and \begin{equation}\label{eq:convergence of ID Cauchy}
	\lim\limits_{n\rightarrow\infty}u_n^*(z)=G_{\mu_{c,\Sigma}}(z)=:u^*(z)
\end{equation} uniformly on $B$. Moreover, $u^*$ can be analytically continued from the upper-half plane to a neighborhood of $\C_+\cup\{z\in\R: -\infty< \Im G(z+\ii0)<0\}$ and we will denote this extension also by $u^*$.

\end{enumerate}
\end{lemma}

For fixed $z$, we may drop its dependence and write $u_n^*(z)=u_n^*$ and $u^*(z)=u^*$.
We are now ready to state our main result on the asymptotics of Appell sequences.

\begin{theorem}\label{thm:Appell Planc-Roa asymp}
	Let $\{f_n\}_{n\geq1}$ be a sequence of functions satisfying Assumption \ref{assump:main assumption} and let $d$ be a fixed integer. \begin{enumerate}
		\item If $z\in\C_{+}$, then  \begin{equation}\label{eq:PR asymp upper half}
			\begin{aligned}
				A_{n+d,f_n}(z)&=\frac{(n+d)!}{n^{n+d}}\sqrt{\frac{(u^*)'(z) }{2\pi n}}\frac{f_{n}\left(nu_n^*(z) \right)\exp\left(nzu_n^*(z) \right)}{u_n^*(z)^{n+d+1}}(1+o(1)).
			\end{aligned}
		\end{equation}
	
		\item If $z\in\R$ and $-\infty<\Im G_{\mu_{c,\Sigma}}(z+\ii0)<0$, then  \begin{equation}\label{eq:PR asymp bulk}
			\begin{aligned}
				A_{n+d,f_n}(z)&=2\Re\left[\frac{(n+d)!}{n^{n+d}}\sqrt{\frac{(u^*)'(z) }{2\pi n}}\frac{f_{n}\left(nu_n^*(z) \right)\exp\left(nzu_n^*(z) \right)}{u_n^*(z)^{n+d+1}}\right](1+o(1)),
			\end{aligned}
		\end{equation}  where $(u^*)'$ is the derivative of the analytic extension of $u^*$ from $\C_{+}$ to $\C_+\cup\{z\in\R: -\infty< \Im G(z+\ii0)<0\}$ as defined in Lemma \ref{lem:clean lem_unique_saddlepoint}.
	\end{enumerate} Moreover, \eqref{eq:PR asymp upper half} and \eqref{eq:PR asymp bulk} hold uniformly on compact subsets of $\C_{+}$ and $\{z\in\R:-\infty< \Im G_{\mu_{c,\Sigma}}(z+\ii0)<0 \}$, respectively. 
\end{theorem} Note that by Stirling's formula, the prefactor satisfies $\frac{(n+d)!}{n^{n+d}}\sqrt{\frac{1}{2\pi n}}\sim e^{-n}$. However, while we only consider fixed $d$ we choose to leave the prefactor in this form as a potential prediction when $d$ is growing with $n$. 

With Proposition \ref{prop:convergence of R-transforms}, Lemma  \ref{lem:clean lem_unique_saddlepoint}, and Theorem \ref{thm:Appell Planc-Roa asymp} in hand we are ready to prove Theorem \ref{thm:main result, general}. \begin{proof}[Proof of Theorem \ref{thm:main result, general}]
 Let $G_n$ be the Cauchy transform of $\llbracket A_{n,f_n}\rrbracket$. It is straightforward to check that because $\{A_{j,f_n}\}_{j\geq 1}$ is an Appell sequence that  \begin{equation}
	 	G_{n}(z)=\frac{1}{n}\frac{A_{n,f_n}'(z)}{A_{n,f_n}(z)}=\frac{A_{n-1,f_n}(z)}{A_{n,f_n}(z)}.
	 \end{equation}For $z\in\C_{+}$ we have by Theorem \ref{thm:Appell Planc-Roa asymp}, \eqref{eq:PR asymp upper half} and Lemma  \ref{lem:clean lem_unique_saddlepoint}  that \begin{equation}\label{eq:convergence of Cauchy for ERM} 
				\lim\limits_{n\rightarrow\infty} G_{n}(z)=\lim\limits_{n\rightarrow\infty} u_n^*(z)(1+o(1))=G_{\mu_{c,\Sigma}}(z). 
	\end{equation} As $A_{n,f_n}$ is a real polynomial it is then trivial to prove a similar limit for $z\in\C_{-}$. Let $B\subset\C_{+}$ be some set with piece-wise $C^1$ boundary and note it follows from Cauchy's integral formula that \begin{equation}
	\llbracket A_{n,f_n}\rrbracket(B)=\frac{1}{2\pi\ii}\oint_{\partial B}G_{n}(z)\dd z.
	\end{equation} It then follows from \eqref{eq:convergence of Cauchy for ERM} that \begin{equation}\label{eq:no mass off the line}
	\begin{aligned}
		\llbracket A_{n,f_n}\rrbracket(B)\rightarrow0\\
	\end{aligned}
		\end{equation} By Helly's selection theorem, every subsequence of $\llbracket A_{n,f_n}\rrbracket$ has a vaguely convergent further subsequence with limit $\rho$, a sub-probability measure. However, from \eqref{eq:no mass off the line} $\rho$ must be supported on $\R$, and for any $\eps>0$ it also follows from \eqref{eq:convergence of Cauchy for ERM} that  there exists $y>0$ such that\begin{equation}
	\rho(\R)\geq \int_{\R}\frac{y^2}{x^2+y^2}\dd \rho(x)=-y\Im G_{\rho}(\ii y)=-y\Im G_{\mu_{c,\Sigma}}(\ii y)=1-\eps.
	\end{equation} Thus, $\rho$ is in fact a probability measure and the sequence $\llbracket A_{n,f_n}\rrbracket$ is tight. In particular, every subsequence of $\llbracket A_{n,f_n}\rrbracket$ has a subsequence that weakly converges to $\rho$ on $\R$ having the same Cauchy transform $G_{\mu_{c,\Sigma}}$, and by unicity of the Cauchy transform we obtain weak convergence $\llbracket A_{n,f_n}\rrbracket \Rightarrow \rho=\mu_{c,\Sigma} $.
\end{proof}

We prove Proposition \ref{prop:convergence of R-transforms} and Lemma \ref{lem:clean lem_unique_saddlepoint}  before proving Theorem \ref{thm:Appell Planc-Roa asymp} in the section below.

\begin{proof}[Proof of Proposition \ref{prop:convergence of R-transforms}]
	For any $u\in\C_{-}$ \begin{equation}
		\begin{aligned}
			-\frac{f_{n}'(nu)}{f_{n}(nu)}&=c_{n}+n\sigma_{n}^2u+\sum_{j=1}^{\infty}\frac{\alpha_{j,n}}{1-n\alpha_{j,n}u}-\alpha_{j,n}\\
			&=c_n+n\sigma_{n}^{2}u+\frac{1}{n}\sum_{j=1}^{\infty}\frac{n\alpha_{j,n}}{1-n\alpha_{j,n}u}-n\alpha_{j,n}\\
			&=c_n-\int_{\C}x\dd\Sigma_{n}(x)+\int_{\C}\frac{u+x}{1-xu}\dd\Sigma_{n}(x).
		\end{aligned}
	\end{equation}  In general, $\supp\left(\Sigma_{n}\right)=\{0\}\cup\{n\alpha_{j,n}\}\subset\C$. Note, by assumption  $f_{n}\in\LP_{M_n}$ for some $0\leq M_n=o(n)$, and we can strengthen this bound on the support to  \begin{equation}\label{eq:supp_inclusion}
		\supp\left(\Sigma_{n}\right)=\{0\}\cup\{n\alpha_{j,n}\}\subset \C\setminus \left(B_{\frac{n}{2M_n} }\left(\frac{-\ii n}{2M_n}\right)\cup B_{\frac{n}{2M_n} }\left(\frac{\ii n}{2M_n}\right) \right),
	\end{equation} where $B_{r}(w)$ is an open disk of radius $r$ centered at $w$. Note that the right hand side of \eqref{eq:supp_inclusion} is precisely the image of $\R+\ii[-M_n/n,M_n/n]$ under the inversion map. Eventually $\frac{1}{u}$ is contained in these open disks away from the support of $\Sigma_{n}$. Thus, $h_{u}(x)=\frac{u+x}{1-xu}$ is a bounded continuous function on $\bigcup_{n\geq N}\supp\left(\Sigma_{n}\right)$ for some $N$, and moreover the bound is uniform for $u$ in compact subsets of $\C_{-}$ with the subspace topology. Thus, by \eqref{eq:Sigma_n assumptions} of Assumption \ref{assump:main assumption} \begin{equation}
		\begin{aligned}
			\lim\limits_{n\rightarrow\infty}-\frac{f_{n}'(nu)}{f_{n}(nu)}&= c+ \int_{\R}\frac{u+x}{1-xu} \dd\Sigma(x), 
		\end{aligned}
	\end{equation} which is the $R$-transform of the $\mu_{c,\Sigma}$. Moreover, this convergence is uniform on compact subsets of $\C_{-}$ by Montel's Theorem.
	
	The proof of \eqref{eq:log of f_n convergence} follows in a similar manner after recalling we define $$\left[ \log(1-xw)\frac{x^2+1}{x^2}+\frac{w}{x}\right]\Big|_{x=0}:=-\frac{w^2}{2}.$$ Specifically, \begin{equation}
		\begin{aligned}
			\frac{1}{n}\log f_n(nu)-\frac{1}{n}\log f_{n}(-n\ii)&=-\frac{n\sigma_n^2}{2}u^2-c_nu+\frac{1}{n}\sum_{j=1}^{\infty}\log\left(1-un\alpha_{j,n}\right)+un\alpha_{j,n}\\
			&\quad+\frac{n\sigma_n^2}{2}(-\ii)^2+c_n\ii-\frac{1}{n}\sum_{j=1}^{\infty}\log\left(1+\ii n\alpha_{j,n}\right)-\ii n\alpha_{j,n}\\
			&=-(u+\ii)\left[c_n-\int x\dd\Sigma_{n}(x)\right]\\
			&\quad +\int_{\C}\left[\log(1-ux)-\log(1+\ii x)\right]\frac{x^2+1}{t^2}+\frac{u+\ii}{x}\dd\Sigma_{n}(x)
		\end{aligned}
	\end{equation} The integrand is continuous away from $x=\frac{1}{u}$ and by the same argument as above is  a continuous bounded function on some domain containing $\bigcup_{n\geq N}\supp\Sigma_{n}$ for some large $N$. \eqref{eq:log of f_n convergence} then follows from Assumption \ref{assump:main assumption}.
	
	Finally, \eqref{eq:cutoff log of f_n convergence} follows in a nearly identical manner to \eqref{eq:log of f_n convergence} after noting that the integrand is continuous and bounded for any $u\in\C_-\cup\R\setminus\{0\}$, and this bound is uniform on compact subsets of $\C_-\cup\R\setminus\{0\}$. 
\end{proof}

\begin{proof}[Proof of Lemma \ref{lem:clean lem_unique_saddlepoint}]
	Let $\mu=\mu_{c,\tilde\Sigma}$ be $\boxplus$-ID 
	and set $J_{\mu}(z)=R_{\mu}(\frac{1}{z})+z$, $F_{\mu}(z)=\frac{1}{G_{\mu}(z)}$, and let $v(x)$ be the unique solution to\begin{equation*}
		\int_{\R}\frac{1+t^2}{(t-x)^2+v(x)^2}\dd\tilde{\Sigma}(t)=1.
	\end{equation*} $J_\mu$ and $F_{\mu}$ are then left and right inverses on the appropriate domains. Additionally, $F_{\mu}$ has a one-to-one continuous extension to $\C_{+}\cup\R$ (which follows from \cite[Theorem 4.6]{Belinschi-Bercovici2005}) and maps $\R$ bijectively to $\{x+\ii v(x):x\in\R\}$ (see \cite{Huang2015} or more specifically \cite[Proposition 2.3]{Bercovici-Wang-Zhong2018}). One can also check explicitly from Proposition \ref{rthm:ID rep} that $\Im J_\mu(x+\ii y)=0$ if and only if $x+\ii y\in\{x+\ii v(x):x\in\R\}$. It then follows that\begin{equation}
	F_{\mu}(\C_{+})=\left\{z\in\C_{+} : \Im J_{\mu}(z)>0 \right\},
	\end{equation} and  \begin{equation}
	\left\{u\in\C_{-} :\Im\left( R_{\mu}(u)+\frac{1}{u}\right)>0 \right\}=G_{\mu}(\C_{+}).
	\end{equation} The existence and uniqueness of a solution to \eqref{eq:general R transform version saddle point.} then follows by noting that \eqref{eq:general R transform version saddle point.} is equivalent to $z=R_{\mu}(u)+\frac{1}{u}$.

	If $f_n\notin\LP$, then we note that \begin{equation*}
		z-\frac{1}{u}-R_n(u)=z-\frac{1}{u}-R_{\mu_{c,\Sigma}}(u)+R_{\mu_{c,\Sigma}}(u)-R_n(u),
	\end{equation*} and by Proposition \ref{prop:convergence of R-transforms} \begin{equation}
	\sup_{u\in B}\left|R_{\mu_{c,\Sigma}}(u)-R_{n}(u) \right|\rightarrow0,
	\end{equation}  for any compact subset $B\subset\C_-$. By Rouche's theorem, the saddle point equation \eqref{eq:clean saddle point equation} has the same number of solutions in any compact subset as \eqref{eq:general R transform version saddle point.} for $n$ sufficiently large. Thus, for any $\eps>0$, \eqref{eq:clean saddle point equation} has at most one solution in $\{u:\Im u<-\eps\}$ for sufficiently large $n$, and a solution exists if $\eps>0$ is sufficiently small.
	
	Let $K_n(u)=\frac{1}{u}+R_n(u)$ and $K(u)=\frac{1}{u}+R_{\mu_{c,\Sigma}}(u)$. By the first two parts of Lemma \ref{lem:clean lem_unique_saddlepoint}, there exists open $D\subset\C_-$  such that $K_n$ is injective on $D$ for sufficiently large $n$ and $B\subset K_n(D)$. It then follows from the open mapping theorem and inverse function theorem (see for example \cite[\S 8.1, Proposition 1.1]{Stein--Shakarchi2003}) that $K_n^{-1}=u_n^{*}$ is holomorphic on $K_n(D)$. Again by convergence of $K_n(u_n^*(z))=z$, each convergent subsubsequence of $u_n^*(z)$ must converge to the unique $u^*(z)$ and it follows that $u_n^*\to u^*=K^{-1}$ uniformly on $B$ by Montel's Theorem. 
	
	Let $u^*(z)=G_{\mu_{c,\Sigma}}(z)=1/F_{\mu_{c,\Sigma}}(u)$. Again, if follows from \cite[Theorem 4.6]{Belinschi-Bercovici2005} that if  $\Im F_{\mu_{c,\Sigma}}(x)>0$ for $x\in\R$, then $F_{\mu_{c,\Sigma}}$ can be analytically continued in a neighborhood of $x$. It then follows that $u^*$ can be analytically continued in some neighborhood of $\{z\in\R: -\infty<\Im G_{\mu}(z+\ii0)<0\}$. 
\end{proof}
	
%

\subsection{The proof of Theorem \ref{thm:Appell Planc-Roa asymp}} We will prove Theorem \ref{thm:Appell Planc-Roa asymp} using the contour integral representation given by \eqref{eq:clean saddle point set up} to apply the saddle point method, similar in spirit to \cite{Hofert}. We will first consider the limiting version of the saddle point equation \begin{equation}\label{eq:clean limiting saddle point equation}
	h'(u)=z-R_{\mu_{c,\Sigma}}(u)-\frac{1}{u}=0.
\end{equation}

Let $u^*(z)\in \C_-$ be the unique solution to \eqref{eq:clean limiting saddle point equation} that is strictly separated away from the real line according to Lemma \ref{lem:clean lem_unique_saddlepoint}. In order to apply the saddle point method to \eqref{eq:clean saddle point set up}, we shall first construct a contour $\Gamma$ encircling the origin along which $u^*(z)$ is the unique maximizer of 
\begin{align}
	H(u):=&
	\lim\limits_{n\rightarrow\infty}\Re\left( h_{n}(u)-h_n(-\ii)\right)\label{eq:H}\\
	=&\Re \big((z-c)u\big)-\frac{\sigma^2}{2}\Re(u^2)-\log|u|+\int\log\left|\frac{1-ut}{1+\ii t}\right|+\frac{\Re(u)t}{t^2+1}\dd\nu(t).\nonumber\\
	&-\Re \big((z-c)(-\ii)\big)-\frac{\sigma^2}{2}.\nonumber
\end{align} 
Recall that $\nu$ is the unweighted L\'evy measure, satisfying $\dd\Sigma(x)=\sigma^2\dd\delta_0(x)+\frac{x^2}{1+x^2}\dd\nu(x)$, cf. Remark \ref{rem:weighted_levy}. 
Once we have the limiting contour we will explain how to adjust the contour in $n$ in order to prove Theorem \ref{thm:Appell Planc-Roa asymp}.

\begin{proposition}\label{prop:Contour}
	For any $z\in\C_+\cup \{z\in\R: -\infty<\Im G_{\mu_{c,\Sigma}} (z+\ii0)<0\}$ there exist a simple curve $\Gamma_-\subseteq \C_-$ connecting $\R_-$ to $\R_+$ and passing through $u^*$ such that $H(u^*)$ is the unique maximum of $H(u)$ among $u\in\Gamma_-$.
\end{proposition}

The exact choice of the contour $\Gamma_-$ is irrelevant for the saddle point method and we are going to deduce abstract existence of $\Gamma_-$ using topological arguments from Morse theory. We refer to \cite{milnor} for general background on the theory. We define the sublevel sets
\begin{align*}
	\sub(t):= \{u\in \mathbb C_-\cup \R : H(u)\le  t\},
\end{align*}
where $t\in\R$ can be interpreted as a height parameter. We will see that $\Gamma_-$ is contained in $\sub(t^*)$, where $t^*=H(u^*)$ is the smallest height such that $\sub(t)$ connects the negative real axis to the positive, see Figure \ref{fig:heatmap}. Later, in the course of the proof of Theorem \ref{thm:Appell Planc-Roa asymp}, we are going to complete $\Gamma_-$ by reflection to $\Gamma_+\subseteq\C_+$, and adjust it for finite $n$ close to the finite $n$ saddle points $u_n^*$.

	\begin{figure}[h]
	\centering
	\includegraphics[width=0.5\linewidth]{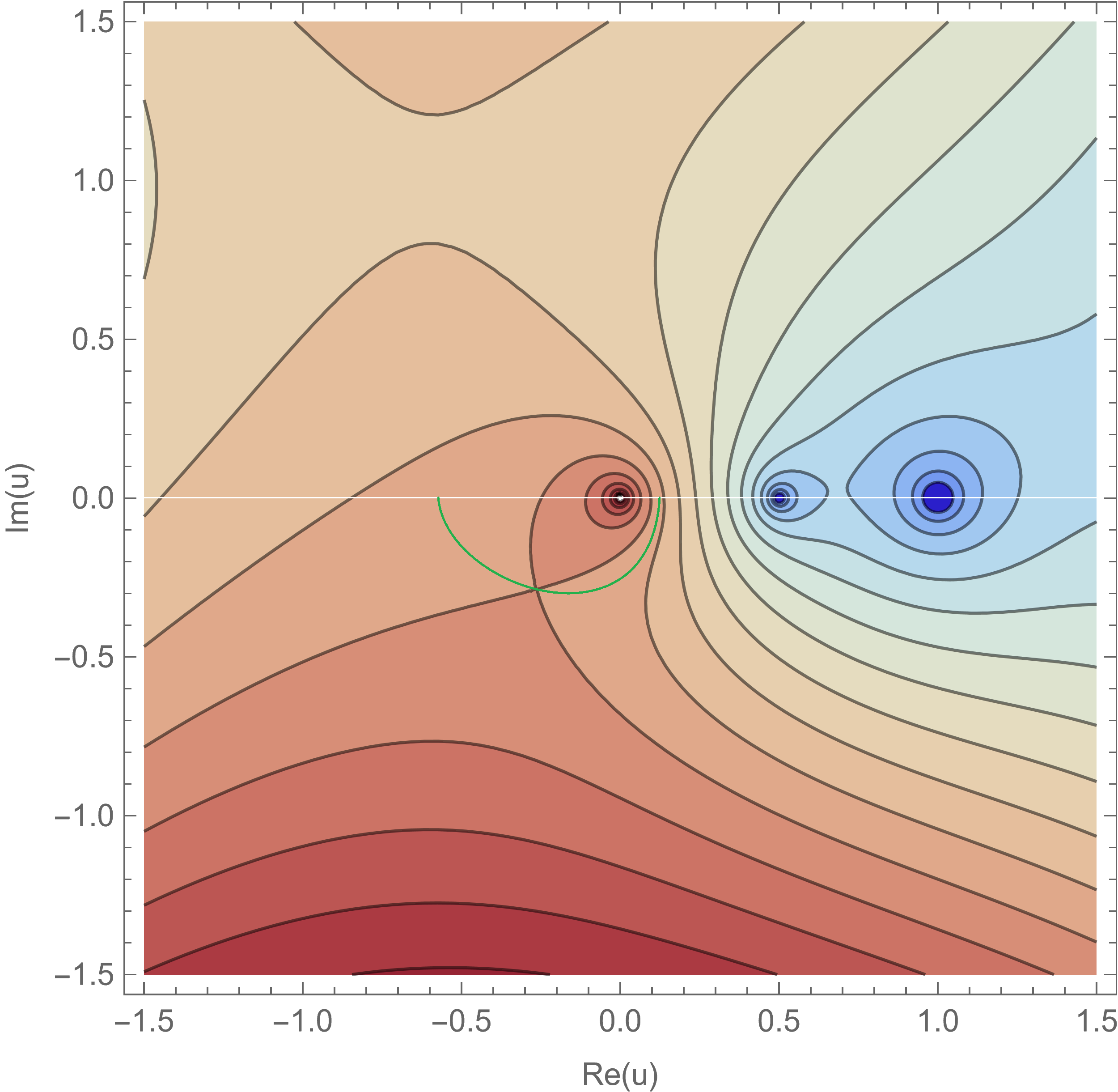}
	\caption{Heatmap of the (sub)level sets of $H(u)$ for $\sigma^2=1$, $z=-1+\ii$, $c=0$, $\nu=\delta_1+\delta_2$, where low values are cold and high values are hot. The unique saddle point $u^*(z)\in\C_-$ is clearly visible such that the green $\Gamma_-$ is contained in $\sub(H(u^*(z)))$ with a unique maximum at $u^*(z)$. Observe also, that there are lower saddle points in $\C_+$ and, close to the singularities on $\R_+$.}
	\label{fig:heatmap}
\end{figure}

We will need that $H$ is monotone decreasing along vertical lines towards $\R$, far away from the origin.
\begin{lemma}\label{lem:H_Monotonicity}
	For any $z\in\C_+$ and $0<\varepsilon<\Im(z)$, there exist $R>0$ such that for all $u\in\C_-$ with $|u|>R$  
	\begin{align*}
		\partial_y H(u)<-\varepsilon,
	\end{align*}
	where $y=\Im(u)$. Moreover, if $z\in\C_+\cup \R$ and $u\in\R\setminus \{0\}$, then $\partial_y H (u-0\ii)=-\Im(z)\le0$. 
\end{lemma} 
\begin{proof}
	For holomorphic $h(u)$ on $\C_-$, it holds $\partial_y\Re h=-2\Im (\partial_u \Re h)=-\Im (\partial_u h)$,
	hence 
	\begin{align*}
		\partial_y H(u)&=-\Im\big(z-c-\sigma^2u-\frac 1 u+\int \frac{-t}{1-ut}+\frac{t}{t^2+1}\dd \nu(t)\big)\\
		&=-\Im z+\sigma^2\Im u-\frac{\Im u}{|u|^2}+\int \frac{t^2\Im u}{|1-ut|^2}\dd \nu(t)<-\Im z-\frac{\Im u}{|u|^2}.
	\end{align*}
	The first claim follows from choosing $|u|>R$ sufficiently large and the second from sending $u\to x\in\R\setminus \{0\}$
\end{proof}

%
%
%

\begin{proof}[Proof of Proposition \ref{prop:Contour}]
	 Observe that $H$ is harmonic in $\C_-$, hence all its critical points are saddle points and by Lemma \ref{lem:clean lem_unique_saddlepoint} there is only one saddle point $u^*$. Moreover, $H$ is continuous up to the real line except for the (positive) singularity at the origin and the (negative) singularity at discontinuity points of $\Sigma$, or $\nu$ respectively.

	Therefore, $\sub(t)$ cannot have bounded holes (or islands) inside $\C_-$ as this would contradict the maximum principle. The same holds up to the boundary $\R\setminus \{0\}$: Assume there would be such a hole, that is a bounded simply connected component of $\sub(t)^c$ which intersects $\R$ in a single interval, then $H|_\R$ must have a local maximum $x^*\in\R\cap \sub(t)^c$. Then, Hopf's Lemma \cite[\S 6.4]{Evans} implies that boundary maxima $H(x^*)$ have strictly positive outer derivative $\partial_y H(x^*)>0$, which contradicts Lemma \ref{lem:H_Monotonicity} unless $x^*=0$.  
	
	Of course, $\sub(t)$ may still enclose a hole around the origin, see Figure \ref{fig:heatmap} for an illustration. In fact, we claim that this enclosure happens precisely at $t^*:=H(u^*)$. 
	
	To verify this, we aim to use the  basic fact from Morse theory that, as $t$ increases, the topology of $\sub (t)$ changes only at critical values of $H$ on $\C_-$, or on the boundary $\partial\C_-=\R$. To formalize this, recall a classical result from Morse theory, sometimes referred to as the \emph{fundamental theorem of Morse theory}, stating, for any smooth function $ h:\mathcal M\to\R$ on a compact manifold $\mathcal M$ without boundary and a non-degenerate critical point $u\in \mathcal M$ of certain index $k$, that sublevel sets $ h^{-1}(-\infty,  h(u)+\varepsilon]$ have the same homotopy type as $ h^{-1}(-\infty,  h(u)-\varepsilon]$ with a $k$-cell attached. See \cite[Theorem 3.2]{milnor} for a proof. In our case, $h:=H$ is harmonic and its critical point $u=u^*$ is a non-degenerate saddle point, thereby having index $k=1$. Hence, attaching $1$-cells translates to merging connected components of $\sub(t)$ at heights $t=H(u^*)$, since attaching both ends of the $1$-cell to the same component creates a loop that is forbidden by the above harmonicity. In order to compactify the domain $\C_-$ of $H$ and to remove singularities on $\R$, we choose $R>0$ from Lemma \ref{lem:H_Monotonicity} and set
	\begin{align*}
		\mathcal M:= \{u\in\C: |u|\le R, \Im(u)\le -\varepsilon\}
	\end{align*} 
	for some $\varepsilon>0$ to be defined below.
	Note the corners near $\pm R-i\varepsilon$ do not affect the argument, and could be smoothed out at scale $\varepsilon$ by redefining $\mathcal M$. Assume $u^*$ lies in the interior of $\mathcal M$, which can be achieved by possibly enlarging $R$. Let $t_0\ll 0$ sufficiently small such that $\sub(t_0)\cap \mathcal M=\emptyset$.  Recall that $H$ is bounded from above on compact subsets away from the origin, and thus we can take $t_1\gg 0$ such that $\sub(t_1)\cap\{u\in\C: \frac{\eps}{2}\le |u|\le R, \Im(u)\leq0\}=\{u\in\C: \frac{\eps}{2}\le |u|\le R, \Im(u)\leq0\}\supset\mathcal M$. The global sublevel set $\sub(t_0)$ does not connect $\R_-$ to $\R_+$, and by Lemma \ref{lem:H_Monotonicity} all connected components of $\sub(t_0)$ intersect the real line $\R$.  By definition,  $\sub(t_1)$ connects $\R_-$ to $\R_+$.
	
	At some $t\in(t_0,t_1)$, this connection within the global sublevel set $\sub(t)$ will be created, but not via the origin, since $0\notin \sub(t)$ for all $t\in\R$. As $t>t_0$ increases, the homotopy type of $\sub(t)\cap \mathcal M$ may also change for heights $t$ that are critical values at the boundary $\partial \mathcal M$ (depending on its normal derivative), as is discussed in \cite[\S 1]{laudenbach}, \cite[p147]{braess} and \cite[\S 3]{Jan}. 
	There are the following possibilities for the connected components of $\sub(t)\cap \mathcal M$:
	\begin{enumerate}
		\item A new component forms on $\partial \mathcal M$, e.g.~at a minimum of $H|_{\R-i\varepsilon}$ on $\R-i\varepsilon$.
		\item Two components merge on $\partial \mathcal M$, e.g.~at a maximum of $H|_{\R-i\varepsilon}$ on $\R-i\varepsilon$.
		\item Components merge at a saddle point in the interior of $ \mathcal M$. 
	\end{enumerate}
	First note that any newly formed or merged component of $\sub(t)\cap \mathcal M$ is still part of a component of the global sublevel set $\sub(t)$, which intersects $\R$. 
	Obviously, (1) cannot introduce a connection from $\R_-$ to $\R_+$ within $\sub(t)$. Recall that there is a unique critical point $u^*\in\mathcal M\subseteq\C_-$, at which two components merge, by (3). However, these components of $\sub(t)$ for $t<t^*$ must necessarily intersect $\R$ on opposite sides, by the above argument using Hopf's Lemma. Thus, at the height $t^*=H(u^*)$, the newly formed connected component of $\sub(t^*)$ connects $\R_- $ to $\R_+ $. In particular, there is a curve $\Gamma_-\subseteq \sub(t^*)$ through $u^*$ as claimed. It can also be argued that the boundary effects (2) do not cause more connections, but we shall not need it.		
	\end{proof}
	
	We now complete the proof of Theorem \ref{thm:Appell Planc-Roa asymp}.
	\begin{proof}[Proof of Theorem \ref{thm:Appell Planc-Roa asymp}]
	The proof of Theorem \ref{thm:Appell Planc-Roa asymp} will proceed using the saddle point method. We follow the approach of \cite{OSullivan2019} which provides a detailed proof of Perron's \cite{Perron17} original argument to compute asymptotic expansions for integrals of the form \begin{equation}
		\int_{\mathcal{C}} e^{np(u)}q(u)\dd u,
	\end{equation}  where $p(u)$ and $q(u)$ are holomorphic on a neighborhood of $\mathcal{C}$, $\Re p(u)$ is maximized at a point $u^*$ of $\mathcal{C}$, and $p,q$ and $\mathcal{C}$ are independent of $n$. There are four obstacles preventing us from simply applying \cite[Corollary 1.4]{OSullivan2019} to \eqref{eq:clean saddle point set up} to complete the proof:
	 \begin{enumerate}
	\item $p=h_n$ and our eventual choice of contour $\mathcal C=\Gamma_n$ depend on $n$, as we will see below.
	\item There are infinitely many saddle points \eqref{eq:clean saddle point equation}, which potentially could be a maximizer (out of which only $u_n^*$ from Lemma \ref{lem:clean lem_unique_saddlepoint} will be dominant as indicated by Proposition \ref{prop:Contour}).
	\item  It is not obvious we can choose our contour $\Gamma$ to maximize $\Re h_n(u)$ at our distinguished saddle point $u_n^*$.
	\item $h_n$ is not holomorphic on any fixed neighborhood of a curve encircling $0$, due to the singularities near $\R$.
	\end{enumerate}
	 These points require some adjustment, but as we will see the uniform convergence of Proportion \ref{prop:convergence of R-transforms} and the fact that we are only interested in the leading order of the asymptotic expansion makes all four manageable. We will provide the outline of \cite{OSullivan2019} and the necessary details to adjust to our application.

	First, we prove that a contour $\Gamma_{n}$ exists such that $\Re h_n(u)< \Re h_{n}(u_n^*)$ for all $u\in\Gamma_{n}\setminus\{u_n^*\}$. Let 
	$$H_n(u):=\Re\left( h_{n}(u)-h_n(-\ii)\right)$$
	with sublevel sets
	$\sub_n(t)=\{u\in\C_{-}\cup \R: H_n(u)\le t\}$ and height $t_n^*=H_{n}(u_n^*)$ converging to $t^*=H(u^*)$. We first claim that \begin{equation}\label{eq:Gamma_{-} in the sublevel set in n}
		\Gamma_-\cap B_{\eps}(u^*)^c\subset \sub_n(t_n^*-\eps''),
	\end{equation} for some $\varepsilon,\eps''>0$ small and $n$ sufficiently large.
	 To prove this, let $h(z)=\lim\limits_{n\rightarrow\infty} h_n(u)-h_n(-\ii)$ be the limit guaranteed by Proposition \ref{prop:convergence of R-transforms} and note that $h''(u^*)\neq 0$. Thus, the Hessian of $H$ at $u^*$ has a strictly negative eigenvalue, and we may take  $\eps,\eps'>0$ to be some small constants such that the curve $\Gamma_-$ from Proposition \ref{prop:Contour} can be chosen such that $\Gamma_-\cap B_{\eps}(u^*)^{c}\subset\sub(t^*-\eps')$.   Let $I=\R\cap\sub(t^*-\eps')$. If $I$ has an isolated point $x_0$, then by Hopf's Lemma \cite[\S 6.4]{Evans} the outer derivative  $\partial_yH(x_0)$ is strictly positive, a contradiction of Lemma \ref{lem:H_Monotonicity}. Thus, by Proposition \ref{prop:Contour} $I$ contains intervals of positive and negative numbers. Next, define \begin{equation}
			H_{M,\infty}(u)=\lim\limits_{n\rightarrow\infty}\Re \left( (u-\ii)z-\log u+\log\ii\right)+\frac{1}{n}\log_{M}| f_{n}(nu)|-\frac{1}{n}\log_{M}|f_{n}(-n\ii)|,
		\end{equation} 
		where we recall $\log_{M}|z|:=\max(\log|z|,-M)$ for $M\in\R$. Note, that both $H_{M,\infty}$ and $H$ are the uniform limits of subharmonic functions, and hence subharmonic themselves by the maximum principle. $\{H_{M,\infty}\}_{M\geq 0}$ is a decreasing sequence of functions. It then follows from the monotone convergence theorem that $H_{M,\infty}\rightarrow H$ point-wise as $M\rightarrow\infty$. Moreover, this convergence is uniform on any compact subset of $\C_-\cup\R\setminus\{0\}$ where $H$ is bounded from below. Let 
		\[\sub_{M}(t):=\{u\in\C_-: H_{M,\infty}(u)\leq t\}. \]
		We take $M$ sufficiently large such that $\Gamma_{-}\cap B_{\eps}(u^*)^{c}\subset \sub_{M}(t^*-\frac{\eps'}{2})$ and such that $\sub_{M}(t^*-\frac{\eps'}{2})\cap\R$ still contains both a positive and negative interval.
		
		 Let $H_{n}(u)=\Re \left(h_{n}(u)-h_n(\ii)\right)$, and let $H_{M,n}$ be the analogous cut-off version. From Proposition \ref{prop:convergence of R-transforms} part (2) we have that $H_n\rightarrow H$ uniformly on compact subsets of $\C_{-}$, and from part (3) that $H_{M,n}\rightarrow H_{M,\infty}$ uniformly on compact subsets of $\C_-\cup\R\setminus\{0\}$.  From Lemma \ref{lem:clean lem_unique_saddlepoint}  the solution $u_n^*(z)$ to \eqref{eq:clean saddle point equation} converges uniformly to $u^*(z)=G_{\mu_{c,\Sigma}}(z)$.  
		
Let $\sub_{M,n}(t)=\{u\in\C_-\cup\R: H_{M,n}(u)\leq t\}$ be the sublevel sets of $H_{M,n}$. Let $B_{r,R}=\{u\in\C_-\cup\R: r\leq |u|\leq R\}$ be such that $\Gamma_{-}\cap B_{\eps}(u^*)^{c}\subseteq B_{r,R}$. For $n$ sufficiently large,\begin{equation}
	\sup_{u\in B_{r,R}}\left|H_{M,n}(u)-H_{M,\infty }(u) \right|<\frac{\eps'}{4}.
\end{equation} Then, for $u\in\Gamma_{-}\cap B_{\eps}(u^*)^{c}$\begin{equation}
H_{M,n}(u)\leq H_{M,\infty }(u)+\left|H_{M,n}(u)-H_{M,\infty }(u) \right|\leq t^*-\frac{\eps'}{4}.
\end{equation}We then have that $\Gamma_{-}\cap B_{\eps}(u^*)^{c}\subset \sub_{M,n}(t^*-\frac{\eps'}{4})\subset \sub_{n}(t^*-\frac{\eps'}{4})$. It follows from Proposition \ref{prop:convergence of R-transforms} and Lemma \ref{lem:clean lem_unique_saddlepoint} that $t_{n}^*=H_{n}(u_n^*)\to t^*$, which implies the first claim \eqref{eq:Gamma_{-} in the sublevel set in n}. For the remainder of the proof we assume that $n$ is large enough that \eqref{eq:Gamma_{-} in the sublevel set in n} holds.

		Following \cite[\S 2]{OSullivan2019} one can compute explicitly, either from the Hessian of $H_n$ or a Taylor expansion of $P_n(r,\theta)=-\frac{\Re(h_n(u_n^*+re^{\ii\theta})-h_{n}(u_n^*) )}{r^2}$, that the level curves $\{u:H_n(u)=H_{n}(u_n^*)\}$ near $u_n^*$ are tangent to the null directions of the Hessian, equivalently to the solutions of $P_n(0,\theta)=0$, given by \begin{equation}
			\theta_{\ell}(0)=-\frac{1}{2}\left(\arg\left(-h_n''(u_n^*)\right)+\frac{(2\ell+1)\pi}{2} \right).
		\end{equation}  Let $\phi^2=\frac{{(u_n^*)'(z)}}{\left| (u_n^*)'(z)\right|}$ and we note that $h_n''(u_n^*)=-\frac{1}{(u_n^*)'(z)}$, since $u_n^*$ is the inverse of $\frac 1 u +R_n(u)$. It follows from \cite[Proposition 2.1]{OSullivan2019} that there exists $\tilde{\eps}_n>0$ such that $u=u_{n}^{*}(z)+\frac{\phi w}{\sqrt{n}}\in\sub_{n}(t_n^*)$ for all $w\in\R$ such that $|u-u_{n}^*|<\tilde\eps_n$, where $\tilde{\eps}_n$ depends only on a finite number of derivatives of $H_n$. Moreover, since $H_{n}$ converges uniformly to $H$, there exists $\tilde{\eps}>0$ independent of $n$ such that $u=u_{n}^{*}(z)+\frac{\phi w}{\sqrt{n}}\in\sub_{n}(t_n^*)$ for all $w\in\R$ such that $|u-u_{n}^*|<\tilde\eps$ for $n$ sufficiently large.

		 We can adjust the contour $\Gamma_{-}$ inside $B_{\eps}(u^*)$ to get a new contour $\Gamma_{-,n}$ such that $u_{n}^*\in \Gamma_{-,n}\subset\sub_n(t_n^*)$. In particular, we deform the part of contour $\Gamma_{-,n}$ inside  $B_{\tilde{\eps}}(u_n^*)\subseteq B_{ {\eps}}(u^*)$ such that $u=u_{n}^{*}(z)+\frac{\phi w}{\sqrt{n}}$ for some $w\in\R$ for $u\in\Gamma_{-,n}\cap B_{\tilde\varepsilon}(u_n^*)$. Finally, we connect these two adjusted pieces along the circle $|u-u_{n}^*|=\tilde{\eps}$. Let $\Gamma_{+,n}$ be the reflection of $\Gamma_{-,n}$ across the real line, and define $\Gamma_{n}:=\Gamma_{-,n}\cup\Gamma_{+,n}$. Observe that for any $u\in\C_{-}$ \begin{equation}
			H_n(\overline{u})\leq H_{n}(u),
		\end{equation} for any $z\in\C_+\cup\R$ with equality if and only if $\Im(z)=0$. 
		
		Let $\Gamma_{-,n,1}=\{u\in \Gamma_{-,n}: |u-u_n^*|< \tilde{\eps}/2\}=\Gamma_{-,n}\cap B_{\tilde\varepsilon/2}(u_n^*)$ and $\Gamma_{-,n,2}=\Gamma_{-,n}\setminus\Gamma_{-,n,1}$. We define $\Gamma_{+,n,1}$ and $\Gamma_{+,n,2}$ as the reflections of these sets across the real line. We take $\Gamma_{n}$ to be the contour in \eqref{eq:clean saddle point set up} and divide the integral into four pieces \begin{equation}
		\begin{aligned}
			\oint_{\Gamma_{n}}\exp\left(n h_n(u) \right)\frac{\dd u}{u^{d+1}}&=\sum_{j\in\{+,-\},k=1,2}\int_{\Gamma_{j,n,k}}\exp\left(n h_n(u)\right)\frac{\dd u}{u^{d+1}}.
		\end{aligned}
		\end{equation} Again, we note that while $h_{n}$ is only piece-wise holomorphic, it holds that $\Re h_n$ is upper-semicontinous, $h_n$ is holomorphic in a neighborhood of $\Gamma_{-,n,1}$ and \begin{equation*}
		\left|e^{nh_{n}(u)}\right|\leq \left|e^{nh_{n}(u_{n}^*)}\right|,
		\end{equation*} for all $u\in\Gamma_{-,n}$, with equality only if $u=u_n^*$.
		
		We first consider the contribution from $\Gamma_{-,n,1}$. Let $\psi_n$ be the holomorphic function such that $h_n(u)=h_{n}(u_n^*)+\frac{1}{2}h_n''(u_n^*)(u-u_n^*)^2(1-\psi_n(u))$ on $B_{\eps}(u^*)$ with $\psi_n(u_n^*)=0$. By Proposition \ref{prop:convergence of R-transforms}, $\psi_n$ converges uniformly to an analogous $\psi$ with $h_n$ replaced by $h$. Hence $\sup_{u\in B_{\eps}(u^*)}|\psi_n(u)|\leq C_{\psi}$ for some $C_{\psi}>0$. This uniform bound is what allows us to adjust \cite{OSullivan2019} to $h_n$ depending on $n$.  By the integral version of Taylor's theorem \cite[Chapter 3.1, Theorem 8]{ahlfors} with $w=\sqrt{N}\phi^{-1}|(u_{n}^*)'(z)|^{-1/2}(u-u_n^*)$ such that $n(h_n(u)-h_n(u_n^*))=-\frac{w^2}{2}(1-\psi_n(u))$ we have for $u\in\Gamma_{-,n,1}$ that \begin{equation}\label{eq:uniform Taylor bound}
			\begin{aligned}
				e^{n(h_n(u) -h_{n}(u_n^*))}\frac{1}{u^{d+1}}&=e^{-w^2/2}e^{w^2\psi_n(u)/2}\frac{1}{u^{d+1}}\\
				&=e^{-w^2/2}\left(\frac{1}{(u_n^*(z))^{d+1}}+ \frac{(u-u_{n}^*)}{2\pi\ii}\int_{|v-u_n^*|=\tilde{\eps}}\frac{e^{w^2\psi_n(v)/2}}{(v-u_n^*)(v-u)}\frac{1}{v^{d+1}}\dd v \right)\\
				&=e^{-w^2/2}\left(\frac{1}{(u_n^*(z))^{d+1}}+ O\left(\frac{e^{C_{\psi}w^2\tilde\eps} }{\tilde{\eps}^2}|u-u_n^*|\right)\right)\\
				&=e^{-w^2/2}\frac{1}{(u_n^*(z))^{d+1}}+O(e^{-w^2/4}|u-u_n^*|),
			\end{aligned} 
		\end{equation} if $\tilde{\eps}$ is sufficiently small, but independent of $n$. Making the change of variables \begin{equation}
		u=u_n^*(z)+\frac{\phi w}{\sqrt{n}}, 
		\end{equation}
		  it follows from \eqref{eq:uniform Taylor bound} that for some $0<c<1/(4|(u_n^*)'(z)|)$, which can be chosen uniformly in $n$ by Proposition \ref{prop:convergence of R-transforms}, that
		  \begin{equation}\label{eq:Gamma_1 change of variables}
		\begin{aligned}
			\int_{\Gamma_{-,n,1}}\exp(nh_{n}(u))\frac{\dd u}{u^{d+1}}&=e^{nh_{n}(u_n^*)}\frac{\phi}{\sqrt{n}}\int_{\Gamma_{-,n,1}'}\frac{\exp\left(-\phi^2\frac{1}{(u_n^*)'\left(z\right)}\frac{w^2}{2}\right)+O(e^{-cw^2})n^{-1/2}}{\left(u_n^*(z) \right)^{d+1}} {\dd w}\\
			&=e^{nh_{n}(u_n^*)}\frac{\phi}{\sqrt{n}}\int_{\Gamma_{-,n,1}'}\frac{\exp\left(-\frac{w^2}{2|(u_n^*)'\left(z\right)|}\right)+O(e^{-cw^2})n^{-1/2}}{\left(u_n^*(z)\right)^{d+1}} \dd w, 	
		\end{aligned}
		\end{equation}where $\Gamma_{-,n,1}'$ is the image of $\Gamma_{-,n,1}$ in the $w$ plane, and recall for the choice of $\Gamma_{-,n}$ that $\Gamma_{-,n,1}'$ is a real interval $[-\sqrt{n}\tilde\eps/2,\sqrt{n}\tilde{\eps}/2]$. We note that again by Lemma  \ref{lem:clean lem_unique_saddlepoint} that $(u_n^*)'(z)$ converges to $G_{\mu_{c,\Sigma}}'(z)$ which is bounded away from $0$ for both $z\in\C_{-}$ and $z\in\R$ such that $-\infty<\Im G_{\mu_{c,\Sigma}}(z+\ii0)<0$, which follows from (for example) \cite[Lemma 2.6]{Bercovici-Wang-Zhong2018}. Hence, \begin{equation*}
		\sup_{w\in[-\sqrt{n}\tilde{\eps}/2,\sqrt{n}\tilde{\eps}/2]}\left|\frac{\exp\left(-\frac{w^2}{2|(u_n^*)'\left(z\right)|}  \right)+O(e^{-cw^2})n^{-1/2}}{\left(u_n^*(z)\right)^{d+1}}  \right|\leq C\left| \frac{\exp\left(-\frac{w^2}{2|(u^*)'(z)|} \right)}{(u^*(z))^{d+1}}\right|,
		\end{equation*} for some constant $C>0$. We may then use the dominated convergence theorem to conclude that \begin{equation}\label{eq:main contour contribution}
	\begin{aligned}
			\int_{\Gamma_{-,n,1}}\exp(nh_{n}(u))\frac{\dd u}{u^{d+1}}&=e^{nh_{n}(u_n^*)}\frac{\phi}{\sqrt{n}}\int_{\Gamma_{-,n,1}'}\frac{\exp\left(-\frac{w^2}{2|(u_n^*)'\left(z\right)|} \right)+O(e^{-cw^2})n^{-1/2}}{\left(u_n^*(z)\right)^{d+1}} \dd w\\
			&=e^{nh_{n}(u_n^*)}\frac{\phi}{(u_n^*(z))^{d+1}\sqrt{n}}(1+o(1))\int_{-\infty}^{\infty}\exp\left(-\frac{w^2}{2|(u^*)'(z)|} \right)\dd w \\
			&=\sqrt{-\frac{2\pi (u^*)'(z)}{n}}\frac{f_{n}\left(nu_n^*(z) \right)\exp\left(nzu_n^*(z) \right)}{u_n^*(z)^{n+d+1}}(1+o(1)).
	\end{aligned}
		\end{equation} 
		\smallskip
		
		We next show the contribution from $\Gamma_{-,n,2}$, and hence also $\Gamma_{+,n,2}$, is negligible. Since $\Gamma_{-,n,2}$ is fixed away from the saddle point and $H_{n}$ converges uniformly as in \eqref{eq:H}, we have $\Re h_{n}(u)-\Re h_{n}(u_n^*)<-c$ for some $c>0$ and all $u\in\Gamma_{-,n,2}$ by Proposition \ref{prop:Contour}. Thus, \begin{equation}
		\begin{aligned}
			\left|\int_{\Gamma_{-,n,2}}\exp(nh_{n}(u))\frac{\dd u}{u^{d+1}}\right|&\leq \int_{\Gamma_{-,n,2}}\exp(n\Re h_n(u))\frac{\dd u}{|u|^{d+1}}\\
			&<\exp({n\Re(h_n(u_n^*))})\int_{\Gamma_{-,n,2}}\exp(-c n) 
			\frac{\dd u}{|u|^{d+1}}\\
			&=o\left(\left|\int_{\Gamma_{-,n,1}}\exp(nh_{n}(u))\frac{\dd u}{u^{d+1}}\right| \right),
		\end{aligned}
		\end{equation}where we used that the length of $\Gamma_{-,n}$ is bounded in $n$.
		
		 For the contribution of $\Gamma_{+,n,1}$ we have two cases: First, if $\Im z>0$, then a similar argument shows that the contribution is negligible, as $\Re h_{n}(\overline{u_n^*})-\Re h_{n}(u_n^*)=\Im z\Im u_n^*<0$. In conclusion, using the contour integral representation \eqref{eq:clean saddle point set up} of Appell polynomials, we obtain the claim
		\begin{align*}
				A_{n+d,f_n}(z) &=\frac{(n+d)!}{2\pi\ii n^{n+d}}\oint_{\Gamma_n}\exp\left(nh_n(u)\right)\frac{\dd u}{u^{d+1}}\\
				&=\frac{(n+d)!}{n^{n+d}}\sqrt{\frac{(u^*)'(z) }{2\pi n}}\frac{f_{n}\left(nu_n^*(z) \right)\exp\left(nzu_n^*(z) \right)}{u_n^*(z)^{n+d+1}}(1+o(1)).
		\end{align*}
	
		Second, if $z\in\R$, then $H_n(\bar u)=\overline{H_n(u)}$ and hence, we have two dominant saddle points $u_n^*(z)$ and $\overline{u_n^*(z)}$ along $\Gamma_n$. Additionally, the direction of travel along $\Gamma_n$ at $\overline{u_n^*(z)}$ is the negative conjugate of the direction of travel at $u_n^*(z)$. Thus, \begin{equation}
			\int_{\Gamma_{-,n,1}}\exp(nh_{n}(u))\frac{\dd u}{u^{d+1}}=-\overline{\int_{\Gamma_{+,n,1}}\exp(nh_{n}(u))\frac{\dd u}{u^{d+1}} }, 
		\end{equation} and \begin{equation}
		\frac{(n+d)!}{2\pi\ii n^{n+d}}\int_{\Gamma_{-,n,1}}\exp(nh_{n}(u))\frac{\dd u}{u^{d+1}}=	\overline{\frac{(n+d)!}{2\pi\ii n^{n+d}}\int_{\Gamma_{+,n,1}}\exp(nh_{n}(u))\frac{\dd u}{u^{d+1}} }.
		\end{equation}  
		Hence the contributions of integrals around the two  dominating conjugate saddle points to $A_{n,f_n}(z)$ are conjugates of each other. This completes the proof.
	\end{proof}

	\section{Proofs of the corollaries}\label{sec:cors} In this section we prove the various corollaries of Theorems \ref{thm:main result, general} and \ref{thm:Appell Planc-Roa asymp}. 
	
		\subsection{Proof of Corollary \ref{cor:fixed f and stable laws} } Fix some function $f\in\LP_M$ for some $M\geq0$. 
		
	If $\sigma^2>0$, let $f_{p}(z)=e^{\frac{\sigma^2}{2}z^2}f(z)$ be the non-Gaussian component of $f$. Then $f_p$ is either a function of order less than $2$ or of order $2$ and type $0$.  Define the finite measure \begin{equation}
		\Sigma_{f_p}=\sum_{j=1}^{\infty}|{\alpha_{j}^2}|\delta_{\alpha_{j}},
	\end{equation} which has no mass at $\{0\}$.  Since $f_n(z)=f(z/\sqrt n)$ we have $\alpha_{j,n}=\alpha_j/\sqrt n$, hence for any bounded  $B\subset\C$ \begin{equation}
	\begin{aligned}
		\left|\Sigma_{n}(B)-\sigma^2\delta_{0}(B)\right|&=\left|\frac{1}{n}\sum_{j=1}^\infty\frac{(\sqrt{n}\alpha_{j})^2}{1+(\sqrt{n}\alpha_{j})^2}\delta_{\alpha_{j}}(B/\sqrt{n})\right|\\
			&\leq \Sigma_{f_p}(B/\sqrt{n}) \to 0,
	\end{aligned}
	\end{equation} by continuity from above of the finite measure $\Sigma_{f_p}$.
	 Additionally, $c_n=O(1/\sqrt{n})$ is trivial. Therefore, Assumption \ref{assump:main assumption} is satisfied with $c=0$ and $\Sigma=\sigma^2\delta_0$, corresponding to the semicircle law. 
	
	If $\sigma^2=0$, let \begin{equation}\label{eq:f_n expanded}
		f_n(z)=e^{-a_nz}f\left(\frac{b_n}{n}z\right)=Ce^{-c\frac{b_n}{n}z-\left(E_n-c\frac{b_n}{n}+\frac{1}{n}\sum_{j=1}^{\infty}\frac{(b_n\alpha_{j})^3}{1+(b_n\alpha_{j})^2}\right) z}\prod_{j=1}^{\infty}\left(1-\alpha_{j}\frac{b_nz}{n}\right)e^{\alpha_{j}\frac{b_n}{n}z},
	\end{equation}  for $f$, $b_n$ and $a_n$ as in \eqref{eq:func def}, \eqref{eq:b_n def} and \eqref{eq:a_n}, respectively.  Then, \begin{equation}\label{eq:a_n expanded}
	\lim\limits_{n\rightarrow\infty} c\frac{b_n}{n}+\left(E_n-c\frac{b_n}{n}+\frac{1}{n}\sum_{j=1}^{\infty}\frac{(b_n\alpha_{j})^3}{1+(b_n\alpha_{j})^2}\right)-\frac{1}{n}\sum_{j=1}^{\infty}\frac{(b_n\alpha_{j})^3}{1+(b_n\alpha_{j})^2}=E,
	\end{equation} which gives the limit in \eqref{eq:c assumptions} of Assumption \ref{assump:main assumption}. 
		
	 It follows from standard properties of regularly varying functions \cite[Proposition 0.8 (v)]{Resnick2007} that $b_n=h(n)n^{1/\alpha}$ for some slowly varying function $h$. The following computation, which proves convergence of $\Sigma_{n}$ to $\Sigma_{\alpha,\theta}$ when $f\in\LP$, will be used in the general case when $f\in\LP_{M}$. Let $f\in\LP$, $r_j=\alpha_{j}^{-1}$ denote the roots of $f$, and $\Lambda_f=\sum_{j}\delta_{r_j}$. It follows from integration by parts and our assumptions that for any $y>0$ \begin{equation}\label{eq:stable measure conv real case}
	 	\begin{aligned}
	 		\Sigma_n([0,y])&=\frac{1}{n}\sum_{j=1}^{\infty}\frac{(b_n\alpha_{j})^2}{1+(b_n\alpha_{j})^2}\delta_{b_n\alpha_{j}}([0,y])\\
	 		&=\frac{1}{n}\sum_{j=1}^{\infty}\frac{b_n^2}{r_j^2+b_n^2}\delta_{r_j}([b_n/y,\infty))\\
	 		&=\frac{1}{n}\int_{b_n/y}^{\infty}\frac{1}{(t/b_n)^2+1 }\dd\Lambda_f(t)\\
	 		&=-\frac{1}{n}\frac{n_{f,+}(b_n/y)}{(1/y)^2+1}+\frac{1}{n}\int_{b_n/y}^{\infty}\frac{2t/b_n^2}{[(t/b_n)^2+1 ]^2}n_{f,+}(t)\dd t\\
	 		&=-\frac{1}{n}\frac{n_{f,+}(b_n/y)}{(1/y)^2+1}+\frac{1}{n}\int_{1/y}^{\infty}\frac{2s}{[s^2+1 ]^2}n_{f,+}(b_ns)\dd s\\
	 		&\sim -\frac{\theta}{y^\alpha}\frac{1}{(1/y)^2+1}+\int_{1/y}^{\infty}\theta\frac{2s^{\alpha+1}}{[s^2+1]^2}\dd s\\
	 		&=\int_{0}^{y}\frac{x^{1-\alpha }}{x^2+1}\theta\alpha \dd x,
	 	\end{aligned}
	 \end{equation} where in the second to last line, we use \eqref{eq:theta def} and the definition of regularly variation, and the final line is another integration by parts and change of variables.  The last line gives \eqref{eq:stable_levy_measure} as claimed, and an identical computation gives the limits of $\Sigma_{n}([-y,0])$, completing the proof for $f_n\in\LP$. 
	 
	 We now show the general case $f\in\LP_{M}$. For any $B\subseteq\C$ we let $\Re(B)$ be the projection to $\R$, i.e.\ $x\in\Re(B)$ if and only if $\{x\}\times\ii\R\cap B\neq \emptyset$. We decompose $\Sigma_{n}$ as two finite signed measures \begin{equation*}
	 	\Sigma_{n}=\Sigma_{n,\R}+\ii\Sigma_{n,\ii\R}.
	 \end{equation*} We next note for any measurable $B$ the total variation of $\Sigma_{n,\ii\R}$ for large $n$ can be bounded by  \begin{equation}
	\begin{aligned}
			 |\Sigma_{n,\ii\R}|(B)&\le \frac{1}{n}\sum_{j=1}^{\infty}\left|\Im\left(\frac{b_n^2}{r_j^2+b_n^2}\right)\right|\delta_{b_n\alpha_{j}}(B)\\
			 &\leq \frac{1}{n}\sum_{j=1}^{\infty}\frac{2|\Re r_{j}|M}{b_{n}^2\left(1+\left[\frac{\Re r_j}{b_n}\right]^2\right)^2 }\delta_{b_n\Re\alpha_{j}}(\Re B ),
	\end{aligned}
	 \end{equation} where the weight is bounded from above and changing from $\delta_{b_n\alpha_{j}}(B)$ to $\delta_{b_n\Re\alpha_{j}}(\Re B )$ over counts the number of $b_n\alpha_{j}$ in the set. An argument with steps identical to \eqref{eq:stable measure conv real case} (using integration by parts and regular variation) shows \begin{equation}
	  \frac{1}{n}\sum_{j=1}^{\infty}\frac{2|\Re r_{j}|M}{b_{n}^2\left(1+\left[\frac{\Re r_j}{b_n}\right]^2\right)^2 }\delta_{b_n\Re\alpha_{j}}\rightarrow0,
	 \end{equation} where we use $0$ to denote the zero measure. Thus, the total variation of $\Sigma_{n,\ii\R}$ converges vaguely to the zero measure, as does $\Sigma_{n,\ii\R}$. 
	 
	 For $\Sigma_{n,\R}$, an identical argument can be used to upper bound  the total variation by \begin{equation}\label{eq:bound on real sigma measure}
	 	\begin{aligned}
	 		|\Sigma_{n,\R}|(B)&= \frac{1}{n}\sum_{j=1}^{\infty}\left|\Re\left(\frac{b_n^2}{r_j^2+b_n^2}\right)\right|\delta_{b_n\alpha_{j}}(B)\\
	 		&\leq \frac{1}{n}\sum_{j=1}^{\infty}\frac{b_n^2}{(\Re r_j)^2+b_n^2}\delta_{b_n\Re \alpha_j}\left(\Re B\right),
	 	\end{aligned}
	 \end{equation} for $n$ sufficiently large so that $b_n>M$. As shown in \eqref{eq:stable measure conv real case} the measure on the right hand side of \eqref{eq:bound on real sigma measure} converges to finite positive measure with no mass at $0$. Thus, \begin{equation}\label{eq:near 0 Sigma bound}
	 |\Sigma_{n}|(B_{\eps}(0))=O_{\eps}\left(\eps^{2-\alpha}\right),
	 \end{equation} where $O_{\eps}$ is the asymptotic notation as $\eps\rightarrow0^{+}$ and the bound can be taken uniformly in $n$. 
	 
	 	We next consider $\Sigma_n$ on sets bounded away from $0$, where we drop the weight function and invert the support to consider the root counting measures\begin{equation}
	 			\Lambda_{f,n}=\frac{1}{n}\sum_{j=1}^{\infty}\delta_{r_{j}/b_n},
	 		\end{equation} on bounded sets. Recall $n_{f,+}(x)= |\{\alpha_{j}^{-1}=r_j \in (0,x)\times\ii[-M,M]\}|$, hence for any $x>0$, $\eps,\eps'>0$, and $n$ sufficiently large \begin{equation}
	 			\begin{aligned}
	 					\Lambda_{f,n}\left((0,x)\times\ii[-\eps',\eps]\right)&=\frac{1}{n}n_{f,+}\left(x{b_n}\right)\\
	 					&=\frac{1}{n}n_{f}\left(xb_n\right)\frac{n_{f,+}\left(xb_n\right)}{n_{f}\left(x{b_n}\right)}\\
	 					&\sim \theta x^{\alpha}.
	 				\end{aligned}
	 		\end{equation}  Similarly, for any $x<0$ \begin{equation}
	 			\Lambda_{f,n}\left((x,0)\times\ii[-\eps',\eps]\right)\sim (1-\theta){|x|^{\alpha}}.
	 		\end{equation}  Additionally, if $B$ is any closed bounded set such $B\cap\R=\emptyset$, then $\Lambda_{f,n}(B)=0$ for $n$ sufficiently large. Thus, $\Lambda_{f,n}$ converges vaguely to the measure $\Lambda_{\alpha,\theta}$ on $\R$ with density $$\dd \Lambda_{\alpha,\theta}(x)=\alpha\left(\theta\oindicator{x>0}+(1-\theta)\oindicator{x<0}\right)|x|^{\alpha-1}\dd x.$$ 
	 		
	 		We now combine these facts to complete the proof. Let $f$ be a continuous function with compact support. Then, from \eqref{eq:near 0 Sigma bound} we see for any $\eps>0$ \begin{equation}\label{eq:f Sigma bound}
	 			\left|\int_{B_{\eps}(0)}f\dd\Sigma_{n} \right|=O_{\eps}(\eps^{2-\alpha}),
	 		\end{equation} uniformly in $n$. Additionally, \begin{equation}\label{eq:f Sigma unbound}
	 		\begin{aligned}
	 				\int_{B_{\eps}(0)^c}f(t)\dd\Sigma_{n}(t)&=\int_{B_{1/\eps}(0)} f\left(\frac{1}{s}\right)\frac{1}{s^2+1}\dd\Lambda_{f,n}(s)\\
	 				&\rightarrow\int_{B_{1/\eps}(0)} f\left(\frac{1}{s}\right)\frac{1}{s^2+1}\dd\Lambda_{\alpha,\theta}(s)\\
	 				&=	\int_{B_{\eps}(0)^c}f(t)\dd\Sigma_{\alpha,\theta}(t)
	 		\end{aligned}
	 		\end{equation} Combining \eqref{eq:f Sigma bound} and \eqref{eq:f Sigma unbound} completes the proof.

	\subsection{Proof of Theorem \ref{thm:local spacing}} Theorem \ref{thm:Appell Planc-Roa asymp} provides the asymptotic equivalence
		\begin{align}\label{eq:PR asymp in bulk}
		A_{n+d,f_n}(z)\sim 2\Re\bigg[&\frac{(n+d)!}{n^{n+d}}\sqrt{\frac{(u^*)'(z) }{2\pi n}}\frac{f_{n}\left(nu_n^*(z) \right)\exp\left(nzu_n^*(z) \right)}{u_n^*(z)^{n+d+1}}\bigg]\nonumber\\
		\sim 2\Re\bigg[&\frac{(n+d)!}{n^{n+d} u^*(z)^{d+1}}\sqrt{\frac{(u^*)'(z) }{2\pi n}}\nonumber\\
		&\exp\left(n\left(  zu_n^*(z)+\frac{1}{n}\log f_{n}(nu_n^*(z))-\log u_n^*(z) \right)\right)
		\bigg]
	\end{align}
	 where $z\in\R$ such that $u^*(z)=G_{\mu_{c,\Sigma}}(z)$ and $\infty>-\Im G_{\mu_{c,\Sigma}}(z+\ii0)>0$. For $z=E+\frac{w}{n}$, we aim to Taylor expand  \eqref{eq:PR asymp in bulk} about $E$. Focusing only on the terms which depend on $w$ and $d$,
	we obtain \begin{align}
		&n\left(  zu_n^*(z)+\frac{1}{n}\log f_{n}(nu_n^*(z))-\log u_n^*(z) \right)\nonumber \\
		=\ & n\left(Eu_n^*(E)+\frac{1}{n}\log f_{n}(nu_n^*(E))-\log u_n^*(E) \right)\label{eq:somemiddleterm} \\
		&+wu_n^*(E)+w(u_n^*)'(E)\left[E+\frac{f_n'(nu_n^*(E))}{f_n(nu_n^*(E))}-\frac{1}{u_n^*(E)}\right]+o(1),\nonumber
	\end{align} where the second derivative of $\left(zu_n^*(z)+\frac{1}{n}\log f_{n}(nu_n^*(z))-\log u_n^*(z)\right)$ can be uniformly bounded in a neighborhood of $E$ using Proposition \ref{prop:convergence of R-transforms} and Lemma \ref{lem:clean lem_unique_saddlepoint}. However, $u^*_{n}(E)$ is a solution to the saddle point equation \eqref{eq:clean saddle point equation}, hence \begin{align*}
	E+\frac{f_n'(nu_n^*(E))}{f_n(nu_n^*(E))}-\frac{1}{u_n^*(E)}=0.
	\end{align*} Therefore, the only $w$ dependent term in \eqref{eq:PR asymp in bulk} is $wu_n^*(E)\to wG(E+\ii0)$, which for simplicity we denote as $wG(E)$. The terms of \eqref{eq:PR asymp in bulk} and \eqref{eq:somemiddleterm} which depend only on $n$, $E$, and $\mu_{c,\Sigma}$ are absorbed in the constants
	\begin{align}
		 &2\frac{(n+d)!}{n^{n+d} |G(E)|}\sqrt{\frac{|G'(E)| }{2\pi n}}e ^{  n\Re\left(Eu_n^*(E)+\frac{1}{n}\log f_{n}(nu_n^*(E))-\log u_n^*(E) \right)}\nonumber \\
		 &\sim2\frac{\sqrt{ |G'(E)| }}{ |G(E)|}e ^{  n\Re\left(Eu_n^*(E)+\frac{1}{n}\log f_{n}(nu_n^*(E))-\log u_n^*(E) -1\right)}=:C_{E,n}\label{eq:CEn}
		\end{align}
		by an application of Stirling's formula, and 
		\begin{align}
		 B_{E,n}&:=-\arg G(E)+\tfrac 1 2 \arg G'(E)+ n\Im\left(Eu_n^*(E)+\tfrac{1}{n}\log f_{n}(nu_n^*(E))-\log u_n^*(E) \right).\label{eq:BEn}
		\end{align}
	Combining the above, we obtain our claim
		    \begin{align*}
		A_{n+d,f_n}(z)\sim C_{E,n}|G(E)|^{-d}e^{w\Re G(E)}\cos\left(B_{E,n}+w\Im G(E)-d\arg G(E) \right).\qed
	\end{align*}

	\begin{example}\label{ex:Hermite PR check}	
	For Hermite polynomials normalized so that the empirical measure converges to the standard semicircle distribution, we have $u_n^*(z)=u^*(z)=G(z)=\frac{1}{2}\left(z-\sqrt{z^2-4}\right)$  and $f_{n}(z)=e^{-\frac{z^2}{2n}}$. One can then check that \begin{equation*}
		\begin{aligned}
			|G(E)|&=1\\
			|G'(E)|&=\sqrt{\frac{1}{4-E^2}}\\
			\Re\left(Eu_n^*(E)+\frac{1}{n}\log f_{n}(nu_n^*(E))-\log u_n^*(E) -1\right)&=\frac{1}{4}E^2-\frac{1}{2}.
		\end{aligned}
	\end{equation*} In this case \begin{equation}
	C_{E,n}=2\left(\frac{1}{4-E^2}\right)^{-1/4}e^{n\left(\frac{1}{4}E^2-\frac{1}{2}\right)}.
	\end{equation} Similarly we can work out explicitly that $\arg G(E)=-\arccos(E/2)$, where $\arccos$ takes values in $[0,\pi]$, and $\arg G'(E)=\arcsin(-E/2)$, where $\arcsin$ takes values in $[-\pi/2,\pi/2]$, and thus \begin{equation}
	B_{E,n}=-\frac{E\sqrt{4-E^2}}{4}-\frac{\pi}{4}+\left(n+\frac{1}{2}\right)\arccos\left(\frac{E}{2}\right).
	\end{equation} Thus, Theorem \ref{thm:local spacing} implies that for $|E|<2$ \begin{equation}\label{eq:Hermite PR our version}
	\begin{aligned}
		n^{-n/2}\He_{n+d}\left(\sqrt{n}E+\frac{w}{\sqrt{n}}\right)&=2\left(\frac{1}{4-E^2}\right)^{-1/4}e^{n\left(\frac{1}{4}E^2-\frac{1}{2}\right)}e^{\frac{E}{2}w}\\
		&\quad\times\cos\left(-\frac{E\sqrt{4-E^2}}{4}-\frac{\pi}{4}+\left(n+d+\frac{1}{2}\right)\arccos\left(\frac{E}{2}\right)-\frac{w}{2}\sqrt{4-E^2}\right).
	\end{aligned}
	\end{equation} This matches the known Plancherel--Rotach asymptotics, after some algebra and Taylor expansion, for Hermite polynomials, see for example \cite[\S 18.15(v)]{NIST:DLMF}.
		\end{example}

		\subsection{Proofs of Corollaries  \ref{cor:Free convolution theorem} and \ref{cor:f with a root at 0}}
		Corollaries \ref{cor:Free convolution theorem} and \ref{cor:f with a root at 0} follow from combining previous work in finite free probability with our results. 
	\begin{proof}[Proof of Corollary \ref{cor:Free convolution theorem}]
		First we note that if $t>0$, set $f_{n,t}(z)= f_{\lfloor tn\rfloor }(tz)$, such that $c_{n,t}$, and $\Sigma_{n,t}$ are defined from $f_{n,t}$, then 
		 \begin{equation}\label{eq:time change limit computation}
			\begin{aligned}
				c_{n,t}&=tc_{ \lfloor tn\rfloor }\\ \Sigma_{n,t}&=nt^2\sigma_{\lfloor tn\rfloor}^2\delta_{0}+\frac{1}{n}\sum_{j=1}^{\infty}\frac{n^2t^2\alpha_{j,\lfloor tn\rfloor}^2}{n^2t^2\alpha_{j,\lfloor tn\rfloor}^2+1}\delta_{nt\alpha_{j,\lfloor tn\rfloor }}\\
				&=t\left((tn)\sigma_{\lfloor tn\rfloor}^2\delta_{0}+\frac{1}{tn}\sum_{j=1}^{\infty}\frac{(tn)^2\alpha_{j,\lfloor tn\rfloor}^2}{(tn)^2\alpha_{j,\lfloor tn\rfloor}^2+1}\delta_{(tn)\alpha_{j,\lfloor tn\rfloor }} \right)\to t\Sigma
			\end{aligned}
		\end{equation} 
		vaguely as $n\to\infty$.
		It follows from Theorem \ref{thm:main result, general} that $\llbracket A_{n,f_{n,t}}(z)\rrbracket
		\Rightarrow\mu_{tc,t\Sigma} = \mu_{c,\Sigma }^{\boxplus t}$. 		
		After recalling that $f_{\lfloor tn\rfloor}(t\partial_z)p_n(z)=p_n\boxplus_{n} A_{n, f_{n,t}}(z)$, Corollary \ref{cor:Free convolution theorem} follows immediately from Theorem \ref{thm:main result, general} and Lemma \ref{lem: convolution convergence}.
	\end{proof}

	\begin{proof}[Proof of Corollary \ref{cor:f with a root at 0}]
		Note that \begin{equation}
			\tilde{f}_n(\partial_z)p_n(z)=\partial_{z}^{m_n}\left[p_n\boxplus_{n} A_{n,f_n}(z)\right].
		\end{equation} In \cite[Theorem 7.7]{Arizmendi-Fujie-Perales-Ueda2024} the authors prove that for a sequence of real rooted polynomials $q_n$ such that $\llbracket q_n\rrbracket\Rightarrow\mu$ and $m_n$ such that $\frac{m_n}{n}\rightarrow t\in[0,1)$ that \begin{equation}\label{eq:limits of repeated diff}
		\llbracket\partial_z^{m_n}q_n\rrbracket\Rightarrow\mathcal{D}_{1-t}\left(\mu^{\boxplus \frac{1}{1-t}}\right).
		\end{equation} This unconditional statement is the completion of a series of papers on repeated differentiation and free convolution semigroups. 
		Corollary \ref{cor:f with a root at 0} then follows from Corollary \ref{cor:Free convolution theorem} and \eqref{eq:limits of repeated diff}.
	\end{proof}

	\section{The proofs for P\'olya--Schur type operators}   \label{sec:Polyaschur}
		Recall that for any linear $T$ which commutes with $z\partial_{z}$, we have from Definition \ref{def:finite boxtimes} that \begin{equation}\label{eq:multi finite free with f}
		Tp_n(z)=p_n(z)\boxtimes_{n}T(z-1)^{n}.
	\end{equation}  
	\subsection{The proof of Theorem \ref{thm:Polya schur and multiplicative groups} } 
	We will work under the assumption that $t=1$. Extending to general $t>0$ is straightforward and similar to what we have seen in the proof of Corollary \ref{cor:Free convolution theorem}, and the statement is a tautology when $t=0$. Let $A_{n,f_n}^{*}(z)=f_n(z\partial_{z})(z-1)^{n}$. 
	
	It follows from Laguerre's Theorem (Lemma \ref{lem:Lag theorem}) that $A_{n,f_n}^{*}$ has only real roots. Additionally, it follows from  $z\partial_z z^k= kz^k$ that  \begin{equation}\label{eq:A*}
		A_{n,f_n}^{*}(z)=(-1)^{n}\sum_{k=0}^{n}(-1)^{k}\binom{n}{k}f_n(k)z^k.
	\end{equation} Its coefficients alternate, hence all the roots must be  non-negative and since $f_n(0)\neq 0$ all roots of $A_{n,f_n}^{*}$ are positive.
	Thus, to prove Theorem \ref{thm:Polya schur and multiplicative groups} it is sufficient by \eqref{eq:multi finite free with f} and Lemma \ref{lem: mult convolution convergence} to prove that \begin{equation}\label{eq:An*}
	\llbracket A_{n,f_n}^{*}\rrbracket\Rightarrow\tilde\mu_{c,\Sigma}.
	\end{equation}  
	To this end, we make use of the following. \begin{lemma}[Proposition 5.1 in \cite{Jalowy-Kabluchko-Marynych2025part2}]\label{lem:exp prof of conv power}
		Let $Q_{n}$ be a polynomial with real nonnegative roots of the form \begin{equation}
			Q_{n}(z)=\sum_{k=0}^{n}(-1)^{n-k}\binom{n}{k}b_{k,n}^nz^k
		\end{equation} and suppose for every $u\in (0,1)$ \begin{equation}\label{eq:Qprofile}
			\lim\limits_{n\rightarrow\infty} b_{\lfloor un\rfloor,n}=e^{\tilde{g}(u)}
		\end{equation} for some function $\tilde{g}:(0,1)\rightarrow\R$. Then, $\llbracket Q_{n}\rrbracket$ converges weakly to the probability measure $\mu$ on $[0,\infty)$ with $S$-transform \begin{equation}
			S_{\mu}(z)=e^{\tilde{g}'(1+z)},
		\end{equation} for $z\in(-1,0)$.
	\end{lemma}
	
We aim to apply Lemma \ref{lem:exp prof of conv power} to $Q_n=\frac{1}{f_n(n)}A_{n,f_n}^*$ and $b_{k,n}^n=\frac{f_n(k)}{f_n(n)}$. In order to verify \eqref{eq:Qprofile}, we use Proposition \ref{prop:convergence of R-transforms}, part (2). Observe that while Proposition \ref{prop:convergence of R-transforms} is formulated for convergence of $\frac{1}{n}\log f_n(nu)$ for $u\in\C_{-}$, it can be immediately extended to $u\in (0,\infty)$ in our case, where $f_n$ has only negative roots such that $\log f_n(nu)$ is holomorphic in a neighborhood of the positive real line.  Hence, if  $u\in(0,\infty)$, then  \begin{equation}\label{eq:log conv of mult appell}
	\begin{aligned}
		\lim\limits_{n\rightarrow\infty}\frac{1}{n}\log \frac{f_{n}(nu)}{f_n(n)}
		=-c(u-1)+\int_{\R}\log\left(\frac{1-ux}{1-x}\right)\frac{x^2+1}{x^2}+\frac{u-1}{x}\dd\Sigma(x)=:\tilde g (u).
	\end{aligned}
	\end{equation} 
Therefore, $\lim_{n\to\infty} b_{\lfloor un\rfloor,n}=\exp(\frac{1}{n}\log \frac{f_{n}(\lfloor un\rfloor)}{f_n(n)})=\exp(\tilde{g}(u))$ with $\tilde g '=-R_{\mu_{c,\Sigma}}$ by Proposition \ref{prop:convergence of R-transforms} (1). It follows from Lemma \ref{lem:exp prof of conv power} that $\llbracket Q_n\rrbracket=\llbracket A_{n,f_n}^{*}\rrbracket$ converges weakly to a probability  measure $\tilde\mu_{c,\Sigma}$ on $[0,\infty)$, and that\begin{equation}\label{eq:S transform of mult Appell}
	S_{\tilde\mu_{c,\Sigma}}(z)=\exp\left(-R_{\mu_{c,\Sigma}}(1+z) \right),
	\end{equation} for $z\in (-1,0)$. It follows from \cite{Bercovici-Voiculescu1993}, see also \cite[Theorem 2.1]{Arizmendi-Hasebe2018}, that a measure $\mu$ is $\boxtimes$-ID if and only if \begin{equation}
	S_{\mu}\left(\frac{z}{1-z}\right)=\exp\left(v_{\mu}(z)\right),
	\end{equation} for some function $v_{\mu}$ which is analytic on $\C\setminus[0,\infty)$, $v_{\mu}(\bar{z})=\overline{v_{\mu}(z)}$, and $v_\mu(\C_-)\subset \C_+\cup\R$. To see that $
	v_{\tilde\mu_{c,\Sigma}}(z)=-R_{\mu_{c,\Sigma}}\left(\frac{1}{1-z}\right),
$ satisfies these properties one need only note that $R_{\mu_{c,\Sigma}}$ is analytic on $\C\setminus (-\infty,0]$, since $f_n$ has only negative roots, it is a real analytic function by the integral representation \eqref{eq:free levy-khin}, and that $R_{\mu_{c,\Sigma}}$ maps $\C_{-}$ into itself. Thus $\tilde\mu_{c,\Sigma}$ is $\boxtimes$-ID. 
	
	Combining Lemma \ref{lem: mult convolution convergence} and \eqref{eq:S transform of mult Appell} completes the proof of Theorem \ref{thm:Polya schur and multiplicative groups}.  \qed 
	
	During the proof, we recovered the following coefficient asymptotics of $A_{n,f_n}^{*}$ due to  \cite{Jalowy-Kabluchko-Marynych2025,Arizmendi-Fujie-Perales-Ueda2024}, which follows immediately from \eqref{eq:A*} and \eqref{eq:log conv of mult appell} and Stirling's formula. 
	\begin{corollary}
		For $f_n$ as in Theorem \ref{thm:Polya schur and multiplicative groups} and $A_{n,f_n}^{*}(z)= f_n(z\partial_{z})(z-1)^{n}$ with alternating-sign coefficients $a_{k;n}$ we have
		$$\lim_{n\to\infty}\frac 1 n \log \left\lvert\frac{a_{\lfloor un \rfloor+1;n}}{a_{\lfloor un \rfloor;n}}\right\rvert=-\log\left(\frac{1-u}{u}\right)-R_{c,\Sigma}(u)
		$$
		 uniformly in $\varepsilon<u<1-\varepsilon$ for any $\varepsilon>0$.
	\end{corollary}
	
	\subsection{The proof of Theorem \ref{thm:Polya-Schur full version}}   Recall that \begin{equation*}
	\begin{aligned}
			T_np_n(z)&=p_n(z)\boxtimes_{n}T_n(z-1)^{n}\\
						  &=p_n(z)\boxtimes_{n}\sum_{k=0}^{n}(-1)^{(n-k)}\binom{n}{k}\Phi_{T_n}^{(k)}(0)z^{k}\\
						  &=p_n(z)\boxtimes_{n}(-1)^{n}J_{n,\Phi_{T_n}}(-z).\\
	\end{aligned}
	\end{equation*} If $\mu_{c,\Sigma}$ is the limiting measure of $A_{n,f_n}$ for the choice $f_n=\Phi_{T_n}$ in Theorem \ref{thm:main result, general}, then it follows from Corollary \ref{cor:Convergnce of Jensen polynomials} that \begin{equation}
	\left\llbracket (-1)^{n}J_{n,\Phi_{T_n}}(-z) \right\rrbracket\Rightarrow \mathcal{D}_{-1}\mu_{c,\Sigma}^{-1},
	\end{equation} the push-forward of $\mu_{c,\Sigma}^{-1}$ by $x\mapsto -x$. Applying Lemma \ref{lem: mult convolution convergence Fujie} completes the proof.  
	
	\section{Proof of Theorem \ref{thm:rectangular Appell polynomials}}   We will use the power series expansion \begin{equation}
		f_{n}(z)=\sum_{k=0}^{\infty}\eta_{k,n}\frac{z^k}{k!},
	\end{equation} in order to write out the polynomial explicitly. Then, using Definition \ref{def:finite boxtimes} and \eqref{eq:formulas for A and J},	
	\begin{equation}\label{eq:decomp of finite rect inf div}
	\begin{aligned}
		f_n\left(\frac{\lambda}{n}\mathscr{L}_{\beta}\right)z^{n}&=\sum_{k=0}^{n}\eta_{k,n}\frac{\Gamma(n+\beta+1)}{\Gamma(n+\beta-k+1)}\left(\frac{\lambda}{n}\right)^k\binom{n}{k}z^{n-k}\\
		&=\left(\sum_{k=0}^{n}(-1)^{k}\frac{\Gamma(n+\beta+1)}{\Gamma(n+\beta-k+1)}\left(\frac{\lambda}{n}\right)^k\binom{n}{k}z^{n-k}\right)\boxtimes_{n} A_{n,f_n}(z)\\
		&=n!\left(\frac{\lambda}{n}\right)^n(-1)^{n}L_{n}^{(\beta)}\left(\frac{n}{\lambda}z\right)\boxtimes_{n} A_{n,f_n}(z),
	\end{aligned}
	\end{equation} where $L_{n}^{(\beta)}$ is the degree $n$ generalized Laguerre polynomial. Since $p\boxtimes_{n}q$ has non-negative roots if both $p$ and $q$ do, \cite[Theorem 1.16]{Marcus-Spielman-Srivastava2022}, it then follows that $f(\lambda \mathscr{L}_{\beta}/n)z^n$ has positive roots (recall the Appell polynomials of functions in $\mathcal{LPI}$ have non-negative roots). Classical convergence results for zeros distributions of Laguerre polynomials show $\llbracket L_n^{\beta}(nz)\rrbracket\Rightarrow \mathsf{MP}_{1/\lambda,\lambda}$, see for instance \cite[Corollary 4.14]{Jalowy-Kabluchko-Marynych2025part2}, hence $\llbracket L_n^{\beta}(nz/ \lambda)\rrbracket\Rightarrow \mathsf{MP}_{1,\lambda}$ as claimed.  Convergence of $\left\llbracket f_{n}(\mathscr{L}_{\beta})z^{n}\right\rrbracket$ then follows from Theorem \ref{thm:main result, general} and Lemma \ref{lem: mult convolution convergence}.

\appendix 

 \section{Proof of the P\'olya--Schur theorem for rectangular differentiation} \label{appendix}
\begin{proof}[Proof of Lemma \ref{lem:PS for rect}]
	The ``if'' direction follows from \eqref{eq:decomp of finite rect inf div} and \cite[Theorem 1.16]{Marcus-Spielman-Srivastava2022}, as discussed in the proof of Theorem \ref{thm:rectangular Appell polynomials}.

	For the ``only if'' direction we follow the original approach of \cite{Schur--Polya1914,Benz1935}. Define the polynomials \begin{equation}\label{eq:T Lag-Appell}
		T[z^n]=p_{T,n}(z).
	\end{equation} It follows from commutativity and a straightforward computation that $T$ is  degree non-increasing,  i.e.~upper-triangular with constant diagonal in the monomial basis $\{1,z,\dots\}$. Thus, we write $p_{T,n}$ in the form \begin{equation}
	p_{T,n}(z)=\sum_{k=0}^{n}(-1)^{k}t_{k,n}\frac{\Gamma(n+\beta+1)}{\Gamma(n+\beta-k+1)}\binom{n}{k}z^{n-k}
	\end{equation}  From the commutativity relation \begin{equation}
	\mathscr{L}_{\beta}p_{T,n}(z)=Tn(n+\beta)z^{n-1}=n(n+\beta)p_{T,n-1}(z).
	\end{equation} Comparing coefficients we see that \begin{equation}
	t_{k,n}\frac{\Gamma(n+\beta+1)}{\Gamma(n+\beta-k+1)}\binom{n}{k}(n-k)(n-k+\beta)=t_{k,n-1}\frac{\Gamma(n+\beta)}{\Gamma(n+\beta-k)}\binom{n-1}{k}n(n+\beta),
	\end{equation} and \begin{equation}
	t_{k,n}=t_{k,n-1}
	\end{equation} and hence we drop the $n$ notation and use $t_k$. We will assume for simplicity that $t_0=1$ and we let $J_{T,n}(z)=z^np_{T,n}(1/z)$. If instead $t_{r}$ is the first non-zero coefficient we need only change to $J_{T,n}(z)=z^{n-r}p_{T,n}(1/z)$ and make the obvious adjustments below.  For any $k$ let $x_{1,k},\dots,x_{k,k}$ denote the roots of $p_{T,k}$.
	Then, \begin{equation}
	J_{T,k}(z)=\prod_{j=1}^{k}\left(1-x_{j,k}z\right),
	\end{equation} and \begin{equation}
\begin{aligned}
		t_{1}(k+\beta)k&=\sum_{j=1}^kx_{j,k},\\
		t_{k}\frac{\Gamma(k+\beta+1)}{\Gamma(\beta+1)}&=\prod_{j=1}^kx_{j,k}.
\end{aligned}
	\end{equation} From the arithmetic-mean and geometric-mean inequality \begin{equation}
	t_{k}\leq \Gamma(\beta+1)t_{1}^k\frac{ (k+\beta )^k}{\Gamma(\beta+k+1)}.
	\end{equation} It then follows from Stirling's formula that \begin{equation}\label{eq:t_k bound}
	\left(\frac{t_k}{k!}\right)^{1/k}=O\left(\frac{1}{k}\right).
	\end{equation}	We then have that, \begin{equation}
	\lim\limits_{n\rightarrow\infty}J_{T,n}\left(\frac{z}{n(n+\beta)}\right)=\sum_{k=0}^{\infty}(-1)^{k}t_{k}\frac{z^k}{k!}=:f(z),
	\end{equation}  where the right-hand side defines an entire function by \eqref{eq:t_k bound}. $f$ is the limit of non-negative rooted polynomials, and hence in $\mathcal{LPI}$. It then follows from the definition of $\{t_j\}$ and a straight-forward computation that $T[z^n]=f(\mathscr{L}_{\beta})z^n$.
\end{proof}

	\bibliography{Appell}
	\bibliographystyle{abbrv}

\end{document}